\documentclass[a4paper,11pt,oneside,headsepline,abstracton]{scrartcl}

\usepackage{ifdraft}

\usepackage[utf8]{inputenc}
\usepackage[T1]{fontenc}

\usepackage{lmodern}
\usepackage{tgtermes}
\usepackage{amssymb, amsfonts}

\usepackage{bbm}
\usepackage{xcolor}%, colortbl}
\usepackage{lettrine}
\usepackage{url}

\usepackage{stmaryrd}
\usepackage{mathrsfs}
\usepackage[english]{babel}

\usepackage[draft=false]{scrlayer-scrpage}
\usepackage[margin=7em]{geometry}
\usepackage[compact]{titlesec}
\usepackage{needspace}

\usepackage{amsmath}
\usepackage[thmmarks, amsmath, amsthm]{ntheorem}

\usepackage[style=alphabetic,firstinits,url=false,sortcites=true,maxbibnames=99,backend=bibtex]{biblatex}
\bibliography{bibliography.bib}

\ifdraft{{}}{\usepackage[bookmarks]{hyperref}}

\usepackage{booktabs}
\usepackage{enumerate}
\usepackage{float}
\usepackage{verbatim}
\usepackage{todonotes}

\ifdraft{\date{\today}}{\date{}}

\clearscrheadings
\automark[section]{subsection}

\lohead{\headertitle}
\rohead{\headerauthors}

\newenvironment{plainfootnotes}{
  \deffootnote[0em]{0em}{0em}{}
}{
  \deffootnote[1em]{1.5em}{1em}{\textsuperscript{\thefootnotemark}}
}

\newif\ifsubsectionstylePrefixParagraph
\newif\ifsubsectionstyleRunin

\subsectionstylePrefixParagraphtrue
\subsectionstyleRunintrue

\titleformat{\section}[hang]
{\Large\sffamily\bfseries}
{\thesection\hspace{0.1em}{|}}{0.25em}{}

\ifsubsectionstyleRunin
\ifsubsectionstylePrefixParagraph

\titleformat{\subsection}[runin]
{\normalfont\bfseries}
{\S\hspace{.1em}\thesubsection}{0.35em}{}[.\hspace*{.5em}]

\else

\titleformat{\subsection}[runin]
{\normalfont\bfseries}
{\thesubsection\hspace{.1em}|}{0.25em}{}[.\hspace*{.5em}]

\fi
\else
\ifsubsectionstylePrefixParagraph

\titleformat{\subsection}[wrap]
{\normalfont\bfseries\selectfont\filright}
{\S\thesubsection}{.35em}{}
\titlespacing{\subsection}
{16pc}{1.5ex plus .1ex minus .2ex}{1pc}

\else

\titleformat{\subsection}[wrap]
{\normalfont\bfseries\selectfont\filright}
{\thesubsection\hspace{.1em}|}{.25em}{}
\titlespacing{\subsection}
{16pc}{1.5ex plus .1ex minus .2ex}{1pc}

\fi
\fi

\ifsubsectionstylePrefixParagraph

\titleformat{\subsubsection}[runin]
{\normalfont\bfseries}
{\S\hspace{.1em}\thesubsubsection}{0.35em}{}[.]

\else

\titleformat{\subsubsection}[runin]
{\normalfont\bfseries}
{\thesubsubsection\hspace{.1em}|}{0.25em}{}[.]
  
\fi

\AtEveryBibitem{\clearfield{doi}}
\AtEveryBibitem{\clearfield{isbn}}
\AtEveryBibitem{\clearfield{issn}}
\AtEveryBibitem{\clearlist{language}}
\AtEveryBibitem{\clearfield{pages}}

\numberwithin{equation}{section}

\newtheorem{theoremcounter}{theoremcounter}[section]

\theoremstyle{plain}

\newtheorem{conjecture}[theoremcounter]{Conjecture}
\newtheorem{corollary}[theoremcounter]{Corollary}
\newtheorem{lemma}[theoremcounter]{Lemma}
\newtheorem{proposition}[theoremcounter]{Proposition}
\newtheorem{theorem}[theoremcounter]{Theorem}

\newtheorem*{maintheorem}{Theorem}
\newtheorem*{maincorollary}{Corollary}

\theoremstyle{definition}

\newtheorem{definition}[theoremcounter]{Definition}

\theoremstyle{remark}

\newtheorem{remark}[theoremcounter]{Remark}

\newtheorem*{mainremark}{Remark}

\let\cal\undefined
\newcommand{\tbf}{\bfseries}

\newcommand{\tit}{\itshape}

\newcommand{\nbd}{\nobreakdash-\hspace{0pt}}

\newcommand{\cal}{\ensuremath{\mathcal}}

\newcommand{\cL}{\ensuremath{\cal{L}}}

\newcommand{\rmF}{\ensuremath{\mathrm{F}}}

\newcommand{\rmJ}{\ensuremath{\mathrm{J}}}

\newcommand{\rmM}{\ensuremath{\mathrm{M}}}

\newcommand{\rmt}{\ensuremath{\mathrm{t}}}

\newcommand{\td}{\tilde}
\newcommand{\wtd}{\widetilde}

\newcommand{\wht}{\widehat}

\newcommand*{\longhookrightarrow}{\ensuremath{\lhook\joinrel\relbar\joinrel\rightarrow}}

\newcommand{\ra}{\ensuremath{\rightarrow}}

\newcommand{\lra}{\ensuremath{\longrightarrow}}
\newcommand{\lhra}{\ensuremath{\longhookrightarrow}}

\newcommand{\lmto}{\ensuremath{\longmapsto}}

\newcommand{\ZZ}{\ensuremath{\mathbb{Z}}}
\newcommand{\QQ}{\ensuremath{\mathbb{Q}}}
\newcommand{\RR}{\ensuremath{\mathbb{R}}}
\newcommand{\CC}{\ensuremath{\mathbb{C}}}

\renewcommand{\Im}{\ensuremath{\mathop{\mathfrak{Im}}}}

\renewcommand{\pmod}[1]{\ensuremath{\;(\mathrm{mod}\, #1)}}

\newcommand{\Hom}{\ensuremath{\mathop{\mathrm{Hom}}}}

\newcommand{\Mat}[1]{\ensuremath{\mathrm{Mat}_{#1}}}
\newcommand{\MatT}[1]{\ensuremath{\mathrm{Mat}^\rmT_{#1}}}
\newcommand{\GL}[1]{\ensuremath{\mathrm{GL}_{#1}}}

\newcommand{\Mp}[1]{\ensuremath{\mathrm{Mp}_{#1}}}
\newcommand{\Sp}[1]{\ensuremath{\mathrm{Sp}_{#1}}}
\newcommand{\Orth}[1]{\ensuremath{\mathrm{O}_{#1}}}

\newcommand{\T}{\ensuremath{\rmt}}
\newcommand{\rT}{\hspace{.1em}\ensuremath{{}^\T}}

\newcommand{\slashdiv}{\ensuremath{\mathop{/}}}
\newcommand{\HS}{\mathbb{H}}

\newcommand{\Aut}{\ensuremath{\mathrm{Aut}}}
\newcommand{\R}{\mathbb{R}}

\newcommand{\Q}{\mathbb{Q}}
\newcommand{\Z}{\mathbb{Z}}

\newcommand{\C}{\mathbb{C}}
\renewcommand{\H}{\mathbb{H}}

\newcommand{\zxz}[4]{\begin{pmatrix} #1 & #2 \\ #3 & #4 \end{pmatrix}}
\newcommand{\abcd}{\zxz{a}{b}{c}{d}}

\newcommand{\kzxz}[4]{\left(\begin{smallmatrix} #1 & #2 \\ #3 & #4\end{smallmatrix}\right) }
\newcommand{\kabcd}{\kzxz{a}{b}{c}{d}}

\renewcommand{\Im}{\operatorname{Im}}

\newcommand{\calO}{\mathcal{O}}

\newcommand{\calX}{\mathcal{X}}

\newcommand{\bs}{\backslash}

\newcommand{\CH}{\operatorname{CH}}

\renewcommand{\MatT}[1]{\ensuremath{\mathrm{Sym}_{#1}}}

\newcommand{\ord}{\ensuremath{\mathrm{ord}}}
\newcommand{\rot}{\ensuremath{\mathrm{rot}}}
\newcommand{\trace}{\ensuremath{\mathrm{trace}}}

\newcommand{\Pic}{\ensuremath{\mathrm{Pic}}}
\renewcommand{\div}{\ensuremath{\mathrm{div}}}
\newcommand{\cl}{\ensuremath{\mathrm{cl}}}

\title{Kudla's Modularity Conjecture and\\Formal Fourier-Jacobi Series}
\author{
Jan Hendrik Bruinier
and
Martin Westerholt-Raum%
{
\footnote{
The first author is partially supported by DFG grant BR-2163/4-1.  The second author was supported by the ETH Zurich Postdoctoral Fellowship Program and by the Marie Curie Actions for People COFUND Program while this paper was written.}}%
}
\newcommand{\headertitle}{Formal Fourier-Jacobi Series}
\newcommand{\headerauthors}{J.~H.~Bruinier, M.~Westerholt-Raum}

\begin{document}

\begin{plainfootnotes}
\maketitle
\end{plainfootnotes}

\thispagestyle{scrplain}

\begin{abstract}
\noindent
We prove modularity of formal series of Jacobi forms that satisfy a natural symmetry condition.  They are formal analogues of Fourier-Jacobi expansions of Siegel modular forms.  From our result and a theorem of Wei Zhang, we deduce Kudla's conjecture on the modularity of generating series of special cycles of arbitrary codimension and for all orthogonal Shimura varieties.

% KEYWORDS and MSC
\vspace{.25em}
%\noindent{\sffamily\tbf vector valued Siegel modular forms \,|\, slopes of Jacobi forms \,|\, formal theta expansions}\\
\noindent{\sffamily\tbf MSC 2010: Primary 11F46; Secondary 14C25}
\end{abstract}

\tableofcontents
\vspace{.5em}

%%%%%%%%%%%%%%%%%%%%%%%%%%%%%%%%%%%%%%%%%%%%%%%%%%
%%% INTRODUCTION

\Needspace*{4em}
\addcontentsline{toc}{section}{Introduction}
\lettrine[lines=2,nindent=.2em]{\tbf F}{ourier} Jacobi expansions
%of Siegel modular forms
are one of the major tools to study Siegel modular forms.  For example, they appeared prominently in the proof of the Saito-Ku\-ro\-ka\-wa conjecture~\cite{andrianov-1979,maass-1979a,maass-1979b,maass-1979c,zagier-1981}.  Also the work of Kohnen and Skoruppa on spinor $L$-functions~\cite{kohnen-skoruppa-1989,kohnen-krieg-sengupta-1995} features Fourier-Jacobi expansions.  We formalize the notion of Fourier-Jacobi expansions by combining two features of Siegel modular forms:  Fourier-Jacobi coefficients are Siegel-Jacobi forms and Fourier expansions of genus~$g$ Siegel modular forms have symmetries with respect to $\GL{g}(\ZZ)$.

For simplicity, we restrict this exposition to classical Siegel modular forms of even weight.  Suppose that $f$ is a Siegel modular form of even weight $k$ for the full modular group $\Sp{2g}(\Z)$. For $l$ with $1\leq l < g$, the Fourier-Jacobi expansion of $f$ is given by
\begin{align}
\label{eq:intro1}
  f(\tau)
=
  \sum_{0 \le m \in \MatT{l}(\Q)}\hspace{-.2em}
  \phi_m(\tau_1, z)\, e(m \tau_2)
\text{,}
\end{align}
where the variable $\tau$ in the genus $g$ Siegel upper half space $\H_g$ is written as
\[
\tau = \zxz{\tau_1}{z}{\rT z}{\tau_2}
\]
with $\tau_1 \in \HS_{g-l}$, $z \in \Mat{g - l,l}(\C)$, and $\tau_2 \in \HS_l$.
The summation in \eqref{eq:intro1} runs through half-integral symmetric positive semi-definite matrices $m$, and we write $e(x) = \exp(2 \pi i \cdot \trace(x))$.
This amounts to the partial Fourier expansion of $f$ with respect to the variable $\tau_2$.
We call $l$ the cogenus of the Fourier-Jacobi series.
%For the moment, we suppose that $l=1$, so that $m\in \Z$.

The weight $k$ transformation law of $f$ implies that the holomorphic functions $\phi_m(\tau_1, z)$ are Siegel-Jacobi forms of weight $k$ and index $m$, see Section \ref{sect:1.3} for details. Such Siegel-Jacobi forms have Fourier expansions
\begin{align}
\label{eq:intro2}
  \phi_m(\tau_1,z)
=
  \sum_{\substack{n\in \MatT{g-l}(\Q)\\
  r\in \Mat{g-l,l}(\Z)}}
  c(\phi_m;n,r) \, e(n \tau_1+\rT r z)
\text{,}
\end{align}
and the coefficients
%$c(f;t)$
of the usual Fourier expansion
\[
f(\tau)=\sum_{0 \le t\in \MatT{g}(\Q)}
  c(f;t)\, e(t \tau)
\]
of $f$ are given by
\begin{align}
\label{eq:intro3}
c\left( f;\, \zxz{n}{\frac{1}{2}r}{\frac{1}{2}\rT r}{m}\right) = c(\phi_m; n,r).
\end{align}
Note that the weight $k$ transformation law of $f$ implies that its Fourier coefficients satisfy the symmetry condition
\begin{align}
\label{eq:intro4}
c(f;\, t) = c(f;\, \rT u t u)
\end{align}
for all $u \in \GL{g}(\ZZ)$.

In the present paper we study formal analogues of Fourier-Jacobi expansions of Siegel modular forms.
Specifically, let $f$ be a \emph{formal} series of Siegel-Jacobi forms as in \eqref{eq:intro1}, given by a family of Siegel-Jacobi forms $\phi_m$ of weight $k$ and index $m$ for
half-integral positive semi-definite matrices $m\in \MatT{l}(\Q)$.  We stress that no convergence assumption is required of such a series.
For $t\in\MatT{g}(\Q)$, we define the formal Fourier coefficients $c(f,t)$ of $f$ by means of the coefficients $c(\phi_m; n,r)$ of the $\phi_m$ and identity \eqref{eq:intro3}.
We call $f$ a {\em symmetric} formal Fourier-Jacobi series (of weight $k$, genus $g$, and cogenus $l$) if its coefficients satisfy
\eqref{eq:intro4}
%$c(f;\, t) = c(f;\, \rT u t u)$
for all $u \in \GL{g}(\ZZ)$.
Denote the corresponding vector space by $\rmF\rmM^{(g,l)}_k$, and write $\rmM^{(g)}_k$ for the space of Siegel modular forms of genus $g$ and weight $k$.
%
%There is a linear map form the space of Siegel modular forms, $\rmM^{(g)}_k$, into %$\rmF\rmM^{(g,l)}_k$ given by the cogenus~$l$ Fourier-Jacobi expansion.
%Our main result states that every symmetric formal Fourier-Jacobi series converges automatically. In other words:
Our main result is as follows.

\begin{maintheorem}[Modularity of Symmetric Formal Fourier-Jacobi Series]
\label{thm:main-theorem-rigidity}
If %$g \ge 2$ and
$1\leq l<g$, the linear map
\begin{gather*}
  \rmM^{(g)}_k
\lra
  \rmF\rmM^{(g,l)}_k\,
\text{,}\quad
  f
\longmapsto
  \sum_{0 \le m \in \MatT{l}(\Q)} \phi_m(\tau_1, z)\, e( m \tau_2),
\end{gather*}
given by the cogenus~$l$ Fourier-Jacobi expansion,
is an isomorphism.
\end{maintheorem}

%In Theorem~\ref{thm:rigidity-general-case} on page~\pageref{thm:rigidity-general-case}, we allow for formal Fourier-Jacobi series whose coefficients are indexed by matrices of size $l \times l$ with $1 \le l < g$.  We call $l$ the cogenus of formal Fourier-Jacobi series.
The main assertion of the theorem is that every symmetric formal Fourier-Jacobi series converges automatically. Since the symplectic group $\Sp{2g}(\Z)$ is generated by the embedded Jacobi group and the discrete Levi factor $\GL{g}(\Z)$, the transformation law then follows immediately.
\begin{mainremark}
Our work also covers the more general case of symmetric formal Fourier-Jacobi series of half-integral weight for representations of the metaplectic double cover~$\Mp{2n}(\ZZ)$ of~$\Sp{2n}(\ZZ)$, see Theorem~\ref{thm:rigidity-general-case} on page~\pageref{thm:rigidity-general-case}.
The case of vector valued weights, i.e., representations of $\GL{g}(\CC)$ occurring in the factor of automorphy, can also be handled by our method.  See Section~\ref{ssec:extension-to-allvector-valued} for details.
\end{mainremark}

\begin{mainremark}
Formal Fourier-Jacobi expansions have first been studied by Aoki~\cite{aoki-2000} in the case of genus~$2$ Siegel modular forms for the full modular group.  Ibukiyama, Poor, and Yuen~\cite{ibukiyama-poor-yuen-2012} studied them in the case of genus~$2$ paramodular forms of level $1$ through~$4$.  The special case~$g=2$ was independently completed by both authors in separate work~\cite{raum-2013a,bruinier-2015}.
\end{mainremark}

Our main application is Kudla's modularity conjecture for Shimura varieties $X$ associated with orthogonal groups of signature $(n,2)$.
As a special case of earlier joint work with Millson, Kudla \cite{kudla-1997}
attached classes in $\CH^g(X)$ of special cycles $Z(t)$ of codimension~$g$ to positive semi-definite matrices~$t \in \MatT{g}(\QQ)$, and considered
their generating series
\begin{gather*}
  A_g(\tau)
=
  \sum_{t \in \MatT{g}(\QQ)} Z(t)\, q^t.
\end{gather*}
He observed that his results with Millson implied that
the analogous (but coarser) generating series
for the images $\cl_{\mathrm{hom}}Z(t)$ of the cycle classes in cohomology is a Siegel modular form of weight $1+n/2$ and genus $g$.
As a result, he asked~\cite{kudla-1997, kudla-2004} whether such a modularity property already held for the series $A_g$ at the level of Chow groups.

%As a special case of earlier joint work with Millson, Kudla \cite{kudla-1997}
%attached special cycles $Z(t)$ of codimension~$g$ to positive semi-definite matrices~$t \in \MatT{g}(\QQ)$.  He observed that his results with Millson implied that
%the generating series
%\begin{gather*}
%  A_g(\tau)
%=
%  \sum_{t \in \MatT{g}(\QQ)} \cl_{\bullet} Z(t)\, q^t
%\end{gather*}
%for the images $\cl_{\mathrm{hom}}Z(t)$ of the cycles in cohomology is a Siegel modular form of weight $1+n/2$ and genus $g$.
%He asked  \cite{kudla-1997, kudla-2004} whether the analogous modularity property already held for the series for the images $\cl_{\mathrm{Chow}}Z(t)$ of the cycles in Chow groups.

Based on his construction of meromorphic modular forms with explicit divisors, Borcherds showed \cite{borcherds-1999} that the $\CH^1(X)$ valued generating series of special divisors is an elliptic modular form.
Building on Borcherds' work, Zhang~\cite{zhang-2009} proved a partial modularity result for the higher codimension cycles (see also \cite{yuan-zhang-zhang-2009}).
He showed that for every  $m\in \MatT{g-1}(\Q)$,
the $m$-th Fourier-Jacobi coefficient of the $\CH^{g}(X)$ valued generating series $A_{g}$ can be written as a finite sum of push forwards of generating series for special divisors on embedded Shimura sub-varieties of codimension $g-1$.
Employing Borcherds' result, he could then deduce that
the $m$-th Fourier-Jacobi coefficient is a Jacobi form of index $m$.

In our notation, Zhang's result states that $A_g$ is a symmetric formal Fourier-Jacobi series of weight $1+n/2$, genus $g$, and cogenus $g-1$.
Combining it with our main theorem,
%\ref{thm:main-theorem-rigidity},
we obtain:
%In conjunction with Zhang's results, Kudla's modularity conjecture follows from our main theorem.
\begin{maincorollary}[Kudla's Modularity Conjecture]
Kudla's modularity conjecture for Shimura varieties of orthogonal type is true.
\end{maincorollary}
\begin{mainremark}
A more detailed statement can be found in Section~\ref{ssec:kudla-conjecture}.
\end{mainremark}

One further appealing consequence of our work is an algorithm to compute Siegel modular forms, sketched in Section~\ref{ssec:compuation}.  The idea is to consider truncated symmetric Fourier-Jacobi series of cogenus~$1$, which we call symmetric Fourier-Jacobi polynomials.  In the simplest case of scalar valued Siegel modular forms we write $\rmF\rmM^{(g)}_{k, \le B}$ for the space of polynomials
\begin{gather*}
  \sum_{0 \le m < B}\hspace{-.2em}
  \phi_m(\tau_1, z)\, e(m \tau_2)
\text{.}
\end{gather*}
We call $B$ the (truncation) precision, and impose a symmetry condition similar to the previous one.  One can deduce effective lower bounds on~$B$ so that the natural projection $\rmF\rmM^{(g,1)}_k \lra \rmF\rmM^{(g)}_{k, \le B}$ is injective.  Moreover, our Main Theorem implies that this projection is onto, if~$B$ is large enough.  Effective results in this direction are not in sight.  However, if the dimension of $\rmM^{(g)}_k$ is known, we obtain an algorithm that computes Fourier expansions of Siegel modular forms in finite time.  Closed formulas for $\dim \rmM^{(g)}_k$ are known in very few cases, but an algorithm to compute $\dim \rmM^{(g)}_k$
%\todo{Actually examples up to genus~$7$ are even included as numerical data.}
has recently been provided in~\cite{taibi-2014}.  Its correctness depends on a certain assumption on $A$-packets, for which we refer the reader to the original paper.
\vspace{.5em}

The proof of our main theorem can be separated into three parts.    We consider the graded algebra $\rmF\rmM^{(g)}_\bullet =\rmF\rmM^{(g,1)}_\bullet$ of symmetric formal Fourier-Jacobi series of cogenus~$1$.  Clearly, it contains the graded ring of Siegel modular forms $\rmM^{(g)}_\bullet$.  Firstly, slope bounds for Siegel-Jacobi forms allow for dimension estimates of $\rmF\rmM^{(g)}_k$.  They, in particular, imply that $\rmF\rmM^{(g)}_\bullet$ is a finite rank module over~$\rmM^{(g)}_\bullet$.
 Consequently, any $f\in \rmF\rmM_k^{(g)}$ satisfies a non-trivial algebraic relation over $\rmM^{(g)}_\bullet$.
Now, viewing $f$ as an element of the completion $\hat {\mathcal{O}}_a$ of the local ring at boundary points $a$ of a regular toroidal com\-pacti\-fication of the Siegel orbifold, one can deduce that $f$ converges in a neighborhood of the boundary. Using the structure of the Picard group of the Siegel orbifold, it can be shown that $f$ has a holomorphic continuation to the whole Siegel half space and therefore converges everywhere.
%
%Secondly, we study elements of $\rmF\rmM^{(g)}_k$ at the boundary of the toroidal compactification of the Siegel orbifold.  Using general theory of analytic algebras, we deduce that they correspond to functions on some open neighborhood of the boundary.  We thus establish that $\rmM^{(g)}_\bullet$ is algebraically closed in $\rmF\rmM^{(g)}_\bullet$.  Combining both arguments yields modularity of symmetric formal Fourier-Jacobi series of cogenus~$1$.
%

To cover general symmetric formal Fourier-Jacobi series, we use induction on the cogenus and a certain pairing of formal Fourier-Jacobi expansions.  Two tools enter that will probably be of independent interest: In Lemma~\ref{la:formal-siegel-phi-operator} and~\ref{la:formal-theta-expansion}, we study formal versions of the Siegel~$\Phi$ operator and the theta expansion of Fourier-Jacobi coefficients.  Both are, as we show, compatible with our definition of symmetric formal Fourier-Jacobi series.
\vspace{.5em}

We start the paper with Section~\ref{sec:preliminaries} on preliminaries on Siegel modular forms, Jacobi forms, Fourier Jaocbi expansions, vanishing orders, and slope bounds.  Section~\ref{sec:formal-fourier-jacobi-expansions} contains the definition of symmetric formal Fourier-Jacobi Series, compatibility statements for the formal Siegel $\Phi$ operator and theta expansions of Fourier-Jacobi coefficients, and an asymptotic estimate of dimensions of~$\dim \rmF\rmM^{(g)}_k$.  In Section~\ref{sec:purely-transcendental-extension}, we prove that the algebra of Siegel modular forms is algebraically closed as a subalgebra of all symmetric formal Fourier-Jacobi series.  We establish modularity in Section~\ref{sec:rigidity-scalar-valued}.  Finally, we discuss some possible generalizations and applications in Section~\ref{sec:extensions-of-this-work}.  Kudla's modularity conjecture, in particular, is deduced in Section~\ref{ssec:kudla-conjecture}.

\vspace{.5em}
\noindent
{\tit Acknowledgement:}
We thank J.-B.~Bost, E.~Freitag, S.~Grushevsky, J.~Kramer, and S.~Kudla for their help.

\section{Preliminaries}
\label{sec:preliminaries}

\subsection{Siegel modular forms}

The Siegel upper half space of genus~$g$ is denoted by~$\HS_g$.  We typically write $\tau$ for an element in there.  The action of the symplectic group $\Sp{2g}(\RR) \subset \GL{2g}(\RR)$ on $\HS_g$ is given by
\begin{gather*}
  \begin{pmatrix}
  a & b \\ c & d
  \end{pmatrix}
  \tau
=
  (a \tau + b) (c \tau + d)^{-1}
\text{.}
\end{gather*}

Define the metaplectic double cover of $\Sp{2g}(\ZZ)$ as the set
\begin{gather*}
  \Mp{2g}(\ZZ)
=
  \Big\{
  (\gamma, \omega) \,:\,
  \gamma=\kabcd \in \Sp{2g}(\ZZ)
  \;\text{and
  $\omega: \HS_g \rightarrow \CC$ with $\omega(\tau)^2 = \det(c \tau + d)$}
  \Big\}
\end{gather*}
with multiplication
\begin{gather*}
  \big(\gamma_1, \omega_1\big)\, \big(\gamma_2, \omega_2\big)
=
  \big(\gamma_1 \gamma_2, (\omega_1 \circ \gamma_2) \cdot \omega_2\big)
\text{.}
\end{gather*}
The canonical projection $\Mp{2g}(\ZZ) \to \Sp{2g}(\ZZ)$ gives rise to an action on $\HS_g$.
Given $k \in \frac{1}{2}\ZZ$ and a finite dimensional representation~$\rho$ of $\Mp{2g}(\ZZ)$, we write $\rmM^{(g)}_k(\rho)$ for the space of Siegel modular forms of weight~$k$ and type~$\rho$, that is, the space of holomorphic functions $f:\HS_g\to V(\rho)$ satisfying
\begin{gather*}
  f\left( \gamma \tau \right)
=
  \omega(\tau)^{2k}
  \rho\left( \gamma,
            \omega\right)\,
  f(\tau)
\end{gather*}
for all $(\gamma, \omega) \in \Sp{2g}(\ZZ)$ (and being holomorphic at the cusp if $g=1$).  Throughout this work, we assume that $\rho$ factors through a finite quotient of $\Mp{2g}(\ZZ)$.  If $\rho$ is trivial, we suppress it from the notation.  The graded algebra of classical Siegel modular forms is denoted by~$\rmM^{(g)}_{\bullet}$.

For a finite index subgroup $\Gamma$ of $\Mp{2g}(\ZZ)$ define~$\rmM_k(\Gamma)$ as the space of holomorphic weight~$k$ Siegel modular forms for the subgroup~$\Gamma$ with trivial representation.
% That is, we require that
% \begin{gather*}
%   f\big( \left(\begin{smallmatrix} a & b \\ c & d \end{smallmatrix}\right) \big)
% =
%   \omega(c \tau + d)^{2k}
%   \rho\big( \left(\begin{smallmatrix} a & b \\ c & d \end{smallmatrix}\right),
%             \omega\big)\,
%   f(\tau)
% \end{gather*}
% holds for all $(\gamma, \omega) \in \Gamma$.

A Siegel modular form $f\in \rmM^{(g)}_k(\rho)$ has a Fourier expansion
\begin{gather*}
  f(\tau)
=
  \sum_{0 \le t \in \MatT{g}(\QQ)} c(f;\, t)\, e(t \tau)
\text{,}
\end{gather*}
with coefficients $c(f;\, t)$ in the representation space of $\rho$, where $0 \le t$ means that $t$ is positive semi-definite.  Here and throughout, we set $e(x) = \exp(2 \pi i \cdot \trace(x))$.

\subsection{Siegel-Jacobi forms}

Siegel-Jacobi forms are functions  on the Jacobi upper half space
\begin{gather*}
  \HS_{g,l}
=
  \HS_g \times \Mat{g, l}(\CC)
\text{.}
\end{gather*}
Given $0 < l \in \ZZ$, define the metaplectic cover of the full Jacobi group
\begin{gather*}
  {\wtd \Gamma}^{(g,l)}
=
  \Mp{2g}(\ZZ) \ltimes \big(\Mat{g, l}(\ZZ) \times \Mat{g, l}(\ZZ)\big)
\text{.}
\end{gather*}
It acts on $\HS_{g,l}$ via
\begin{gather*}
  \left(
  \begin{pmatrix}
  a & b \\ c & d
  \end{pmatrix},\,
  \omega,\,
  \lambda, \mu
  \right)
  (\tau, z)
=
  \big( (a \tau + b) (c \tau + d)^{-1}, (z + \tau \lambda + \mu) (c \tau + c)^{-1} \big)
\text{,}
\end{gather*}
where $\lambda, \mu \in \Mat{g,l}(\ZZ)$, and $\left(\begin{smallmatrix} a & b \\ c & d \end{smallmatrix}\right)$ and $\omega$ are as in the previous section.

In order to define Siegel-Jacobi forms in a convenient way, we make use of the embedding
\begin{align}
\label{eq:embedding-of-jacobi-group}
%  \iota_{\Mp{}}^{(g,l)}:\,
  {\wtd \Gamma}^{(g,l)}
&\lhra
  \Mp{2(g+l)}(\ZZ)
\\\nonumber
  \left(
  \begin{pmatrix}
  a & b \\ c & d
  \end{pmatrix},\,
  \omega,\,
  \lambda, \mu
  \right)
&\lmto
  \left(
  \begin{pmatrix}
    a & 0^{(g, l)} & b & a \mu - b \lambda \\
    \rT \lambda & 1^{(l)} & \rT \mu & 0 \\
    c & 0 & d & c \mu - d \lambda \\
    0 & 0 & 0 & 1^{(l)}
  \end{pmatrix},\;
  \tau \mapsto \omega( \tau )
  \right)
\text{.}
\end{align}
Fix $k \in \frac{1}{2}\ZZ$, $m \in \MatT{l}(\frac{1}{2}\ZZ)$ with integral diagonal entries, and a finite dimensional representation~$\rho$ of ${\wtd \Gamma}^{(g,l)}$.  We say that a holomorphic function $\phi:\HS_{g,l} \to V(\rho)$ is a (Siegel-) Jacobi form of weight~$k$, index~$m$, and type~$\rho$ if
\begin{gather*}
  \phi(\tau, z) \cdot e(m \tau')
\text{,}\quad
  \begin{pmatrix}
  \tau & z \\ \rT z & \tau'
  \end{pmatrix}
  \in
  \HS_{g + l}
\end{gather*}
transforms like a genus $g + l$ Siegel modular form of weight $k$ and type $\rho$ under the image of \eqref{eq:embedding-of-jacobi-group}
%$\iota_{\Mp{}}^{(g, l)}$
(and $\phi(\tau, \alpha \tau + \beta)$ being holomorphic at the cusp for all $\alpha, \beta \in \QQ^l$ if $g=1$).  The space of Jacobi forms of genus~$g$, weight~$k$, type~$\rho$, and index~$m \in \MatT{l}(\QQ)$ will be denoted by~$\rmJ^{(g)}_{k, m}(\rho)$.  We use notation analogous to the case of Siegel modular forms.  A Jacobi form $\phi$ has a Fourier expansion of the form
\begin{gather}
\label{eq:jacobifourier}
  \phi(\tau,z)
=
  \sum_{\substack{n \in \MatT{g}(\QQ)\\ r \in \Mat{g, l}(\QQ)}}
  c(\phi;\, n, r)\, e(t \tau + r\rT z)
\text{.}
\end{gather}

% Theta decomposition

If $m\in \MatT{l}(\Q)$ is a positive definite half-integral matrix, we define the associated vector valued genus $g$ theta series as follows. For $\mu \in D_g(m):=\Mat{g,l}(\ZZ) (2 m)^{-1} \slashdiv \Mat{g,l}(\ZZ)$ we let
\begin{align}
\label{eq:defthetamu}
\theta^{(g)}_{m, \mu}(\tau, z) = \sum_{x\in \mu+\Mat{g,l}(\Z)} e( x m \rT x \tau)e(2xm \rT z).
\end{align}
\begin{comment}
Define the Siegel theta series of genus $g$ and index~$m$ for $\mu \in \frac{1}{2m}\ZZ^g \slashdiv \ZZ^g$ as follows:
\begin{gather*}
  \theta_{m, \mu} (\tau, z)
=
  \sum_{\lambda \in \ZZ^g + \mu}
  \exp\Big( 2 \pi i m \big(
  \lambda \rT \lambda\, \tau
  +
  2 \rT \lambda z
  \big) \Big)
\text{.}
\end{gather*}
\end{comment}
It is a standard result that $(\theta^{(g)}_{m, \mu})_\mu$ transforms like a vector valued Siegel modular form of weight~$\frac{l}{2}$ and type $(\rho^{(g)}_m)^\vee$, where $\rho^{(g)}_m$ is the Weil representation of $\Mp{2g}(\ZZ)$ on $\CC[D_g(m)]$, see for instance pp.~168 of~\cite{runge-1995}.  The transformation law under the Jacobi group implies that any  $\phi\in   \rmJ^{(g)}_{k, m}$ can be uniquely written as a sum
 $\phi = \sum_{\mu} h_\mu(\tau)\, \theta^{(g)}_{m, \mu}(\tau, z)$, where the functions $h_\mu(\tau)$ are the components of a vector valued Siegel modular form with representation $\rho^{(g)}_m$, see \cite{ziegler-1989}. We obtain a map,
called theta decomposition of Jacobi forms (see e.g.~\cite{ziegler-1989}, Section~3),
\begin{align}
\label{eq:theta-decomposition}
  \rmJ^{(g)}_{k, m}
&\lra
  \rmM^{(g)}_{k - \frac{l}{2}}(\rho^{(g)}_m),
\\\nonumber
  \phi
%= \sum_{\mu} h_\mu(\tau)\, \theta_{m, \mu}(\tau, z)
&\longmapsto
  \big( h_\mu \big)
  _{\mu \in D_g(m)}
\text{.}
\end{align}

\subsection{Fourier-Jacobi expansions}
\label{sect:1.3}

Fix $0 \le l \le g$ and write
\begin{align}
\label{eq:tausplit}
  \tau
=
  \begin{pmatrix}
  \tau_1 & z \\
  \rT z & \tau_2
  \end{pmatrix}
\text{,}
\end{align}
where $\tau_1$ is a $(g-l) \times (g-l)$ matrix, $z$ has size $(g-l) \times l$ and $\tau_2$ has size $l\times l$.  Siegel modular forms of weight $k$  allow for a Fourier-Jacobi expansion
\begin{gather*}
  f(\tau)
=
  \sum_{m\in \MatT{l}(\Q)} \phi_m(\tau_1, z)\, e(m \tau_2)
\text{,}
\end{gather*}
where $m$ runs through symmetric positive semi-definite  matrices of size~$l \times l$ and $\phi_m \in \rmJ^{(g - l)}_{k, m}$.  We say this is the Fourier-Jacobi expansion of cogenus~$l$, and call $\phi_m$ the Fourier-Jacobi coefficient of index~$m$.

\subsection{Vanishing orders}

We say that a symmetric matrix $t \in \MatT{g}(\QQ)$ represents $m \in \QQ$, if there is $v \in \ZZ^g$ with $\rT v t v = m$.  For a non-zero Siegel modular form~$f\in \rmM^{(g)}_k(\rho)$ we define the vanishing order by
\begin{gather*}
  \ord\,f
=
  \inf \big\{
  m \in \QQ \,:\,
  \text{$\exists t \in \MatT{g}(\QQ)$
  such that
  $c(f;\, t) \ne 0$ and $t$ represents $m$}
  \big\}
\text{.}
\end{gather*}
In addition, we use the convention that $\ord\, f = \infty$ if $f = 0$.  The order is additive with respect to the tensor product of Siegel modular forms, $\ord\, f\otimes g = \ord\, f + \ord\, g$.

An analogous definition can be made for  Siegel-Jacobi forms. If $\phi\in \rmJ^{(g)}_{k,m}(\rho)$ is non-zero, we put
\begin{gather*}
  \ord\,\phi
=
  \inf \big\{
  m \in \QQ \,:\,
  \text{
  $\exists t \in \MatT{g}(\QQ)$,
  $r \in \Mat{g,l}(\QQ)$
  such that
  $c(\phi;\, t, r) \ne 0$
  and
  $t$ represents $m$}
  \big\}
\text{.}
\end{gather*}

Spaces of modular forms with vanishing order greater than or equal to~$o \in \QQ$ are marked by square brackets, $\rmM_k^{(g)}(\rho)[o]$ and $\rmJ_{k, m}^{(g)}(\rho)[o]$.  Given a Siegel modular form $f \in \rmM_k^{(g)}(\rho)[o]$ of vanishing order~$o$, then its Fourier-Jacobi coefficients $\phi_m$ are zero if~$m < o$.

\subsection{Slope bounds for Siegel modular forms}

%Recall from Section~\ref{sec:preliminaries} that the vanishing order of Jacobi forms is defined by their Fourier-Jacobi expansions.
Here we recall some known results on slope bounds for Siegel modular forms associated to the full modular group.  From this we deduce slope bounds for vector valued Siegel modular forms.

\begin{definition}
\label{def:slope-of-siegel-modular-forms}
We define the slope of a non-zero Siegel modular form~$f$ of weight~$k$ by
\begin{gather*}
  \varrho(f)
=
  \frac{k}{\ord\,f}
\text{.}
\end{gather*}
The minimal slope bound for scalar valued genus~$g$ Siegel modular forms is written as
\begin{gather}
  \varrho_g
=
  \inf_{f \in \rmM^{(g)}_{\bullet} \setminus \{0\}} \varrho(f)
\text{,}
\end{gather}
\end{definition}

As a corollary to work of Eichler and Blichfeld, we find a lower bound on $\rho_g$ for all $g$.
\begin{theorem}
\label{thm:slope-estimate-for-g-ge-9}
We have
\begin{gather*}
  \varrho_g
>
  \frac{\sqrt{3} \pi^3}{2} \Gamma(2 + \tfrac{n}{2})^{-\frac{4}{n}}
\text{.}
\end{gather*}
\end{theorem}
\begin{proof}
In~\cite{eichler-1975}, Eichler found that
\mbox{$
  \varrho_g
>
  2 \sqrt{3} \pi\, \gamma_g^{-2}
$},
and Blichfeld established in his work~\cite{blichfeldt-1929} that
\mbox{$
  \gamma_g
\le
  \frac{2}{\pi} \Gamma(2 + \tfrac{n}{2})^{\frac{2}{n}}
$}.
\end{proof}

\begin{remark}
For $g \le 5$, we know $\varrho_g$ exactly by unpublished work of Weissauer, by Salvati Manni~\cite{salvati-manni-1992}, and by Farkas, Grushevsky, Salvati Manni, Verra~\cite{farkas-grushevsky-salvati-manni-verra-2014}.  We have
\begin{gather*}
  \varrho_1 = 12
\text{,}\quad
  \varrho_2 = 10
\text{,}\quad
  \varrho_3 = 9
\text{,}\quad
  \varrho_4 = 8
\text{,}\quad
\text{and}\quad
  \varrho_5 = \frac{54}{7}
\text{.}
\end{gather*}
\end{remark}

% \begin{theorem}[Eichler, Blichfeldt]
% We combine estimates by Eichler~\cite{eichler-1975} and Blichfeldt~\cite{blichfeldt-1929} in order to bound $\varrho_g$ in terms of~$g$ alone.
% \begin{theorem}[Eichler]
% \label{thm:slope-bound-eichler}
% We have
% \begin{gather*}
%   \varrho_g
% >
%   2 \sqrt{3} \pi\, \gamma_g^{-2}
% \text{,}
% \end{gather*}
% where $\gamma_g$ is the Hermite constant.
% \end{theorem}
% \begin{theorem}[Blichfeldt]
% We have
% \begin{gather*}
%   \gamma_g
% \le
%   \frac{2}{\pi} \Gamma(2 + \tfrac{n}{2})^{\frac{2}{n}}
% \text{.}
% \end{gather*}
% \end{theorem}
% \begin{corollary}
% \label{cor:slope-estimate-for-g-ge-9}
% We have
% \begin{gather*}
%   \varrho_g
% >
%   \frac{\sqrt{3} \pi^3}{2} \Gamma(2 + \tfrac{n}{2})^{-\frac{4}{n}}
% \text{.}
% \end{gather*}
% \end{corollary}

\begin{proposition}
\label{prop:slope-bound-vector-valued}
Assume that $m > k/\varrho_g$.
For any representation~$\rho$ of~$\Mp{2g}(\ZZ)$ factoring through a finite quotient, we have
\begin{gather*}
\rmM^{(g)}_{k}(\rho)[m]
= \{ 0\}
\text{.}
\end{gather*}
\end{proposition}

\begin{proof}
If $\rho$ is the trivial one-dimensional representation, the assertion is an immediate consequence of the definition of $\varrho_g$.

To prove the assertion for general $\rho$, define $d = \big[ \Mp{2g}(\ZZ) \,:\, \ker(\rho)\big]$ as the index of the kernel of $\rho$.  Let $v \in V(\rho)^\vee$ be a linear form on the representation space of~$\rho$.  For any function $f :\, \HS_g \ra V(\rho)$ we denote by $\langle v, f \rangle$ the function $\tau \mapsto v(f(\tau))$.

Let $f \in \rmM^{(g)}_k(\rho)[m]$, and set
\begin{gather*}
  f_v(\tau)
=
  \prod_{\gamma \in \ker(\rho) \backslash \Mp{2g}(\ZZ)}
  \hspace{-1em}
  \langle v, f \rangle \big|_k \gamma
\text{.}
\end{gather*}
Since $\langle v, f \rangle$ is a scalar valued Siegel modular form of weight $k$ for the group $\ker(\rho)$,
 the function $f_v(\tau)$ is a scalar valued Siegel modular form of weight $dk$ for $\Mp{2g}(\Z)$.
Since $\langle v, f \rangle \big|_k \gamma$ can be expressed as $\langle v_\gamma, f \rangle$ for a suitable $v_\gamma \in V(\rho)^\vee$, we find that $f_v \in \rmM^{(g)}_{d k}[d m]$. Hence
the assertion in the scalar case implies that $f_v = 0$
and $\langle v, f \rangle = 0$ for all $v \in V(\rho)^\vee$.
\end{proof}

%%%%%%%%%%%%%%%%%%%%%%%%%%%%%%%%%%%%%%%%%%%%%%%%%%%%%%%%%%%%%

\section{Symmetric Formal Fourier-Jacobi Series}
\label{sec:formal-fourier-jacobi-expansions}

In this section, we define formal expansions whose coefficients are Jacobi forms and which satisfy symmetry conditions inspired by the action of $\GL{g}(\ZZ) \subset \Sp{2g}(\ZZ)$ on Siegel modular forms.  We will often decompose the variable $\tau \in \HS_g$ into three parts as in \eqref{eq:tausplit}.

Let $\rho$ be a representation of $\Mp{2g}(\Z)$ as before. For every $0\leq l<g$ it induces a representation of the Jacobi group by restriction via the embedding
%$\iota_{\Mp{}}^{(g,l)}$
defined in~\eqref{eq:embedding-of-jacobi-group}, which we also denote by $\rho$.

Given a formal series of Jacobi forms
\begin{gather*}
  f(\tau)
=
  \sum_{0 \le m \in \MatT{l}(\QQ)}
  \phi_m(\tau_1, z)\,
  e\big( m \tau_2 \big)
\end{gather*}
with $\phi_m \in \rmJ^{(g - l)}_{k, m}(\rho)$ for all $m$, define its Fourier coefficients as
\begin{gather*}
  c(f;\, t)
=
  c(\phi_m;\, n, r)
\text{,}
\qquad
  t
=
  \begin{pmatrix}
  n & \tfrac{1}{2}r \\
  \tfrac{1}{2}\rT r & m
  \end{pmatrix}
\text{,}
\end{gather*}
where the coefficients $c(\phi_m; \,n,r)$ are given by \eqref{eq:jacobifourier}.

\begin{definition}
\label{def:formal-fourier-jacobi-expansions}
Fix $k \in \frac{1}{2}\ZZ$ and a representation $\rho$ of $\Mp{2g}(\ZZ)$.  For an integer $0 \le l < g$, let
\begin{gather*}
  f(\tau)
=
  \sum_{0 \le m \in \MatT{l}(\QQ)}
  \phi_m(\tau_1, z)\,
  e\big( m \tau_2 \big)
\end{gather*}
be a formal series of Jacobi forms $\phi_m\in \rmJ^{(g-l)}_{k,m}(\rho)$. Given $u \in \GL{g}(\ZZ)$ and a choice $\omega$ of a square root of $\det(u) \in \{\pm 1\}$, set $\rot(u) = \left(\begin{smallmatrix} u & 0 \\ 0 & \rT u^{-1} \end{smallmatrix}\right)$.  If
\begin{gather*}
  c(f;\, t)
=
  \omega^{2k} \rho\big( \rot(u), \omega \big)\, c(f;\, \rT u t u)
\end{gather*}
holds for all $0 \le t \in \MatT{g}(\QQ)$ and all $u \in \GL{g}(\ZZ)$, then we call $f$ a {\em symmetric} formal Fourier-Jacobi series of genus~$g$, cogenus~$l$, weight~$k$, and type~$\rho$.  The $\phi_m$ are called the canonical Fourier-Jacobi coefficients of~$f$.
\end{definition}
The notion of vanishing orders extends to symmetric formal Fourier-Jacobi series in a straightforward way.

We write $\rmF\rmM^{(g,l)}_k(\rho)$ for the vector space of such symmetric formal Fourier-Jacobi series.
%of genus~$g$, cogenus~$l$, weight~$k$, and type~$\rho$.
If $l = 1$, we abbreviate this by $\rmF\rmM^{(g)}_k(\rho)$.  Further, set
\begin{gather*}
  \rmF\rmM^{(g)}_\bullet(\rho)
=
  \bigoplus_k \rmF\rmM^{(g)}_k(\rho)
\text{.}
\end{gather*}
If $\rho_0$ is the trivial representation on $\CC$, we briefly write $\rmF\rmM^{(g)}_\bullet =\rmF\rmM^{(g)}_\bullet(\rho_0)$.

\begin{proposition}
\label{prop:module-and-algebra-structure-for-formal-fourier-jacobi-expansions}
Scalar valued symmetric formal Fourier-Jacobi series~$\rmF\rmM^{(g)}_{\bullet}$ carry an algebra structure over the graded ring of classical modular forms $\rmM^{(g)}_{\bullet}$.  Symmetric formal Fourier-Jacobi series~$\rmF\rmM^{(g)}_\bullet(\rho)$ are a module over~$\rmM^{(g)}_{\bullet}$.
\end{proposition}
\begin{proof}
This amounts to a straightforward verification of the symmetry condition of Definition~\ref{def:formal-fourier-jacobi-expansions}.
\end{proof}

\subsection{The Siegel $\Phi$ operator and Fourier-Jacobi coefficients}
\label{ssec:siegel-phi-and-fourier-jacobi}

We describe two results, which allow us to reduce considerations of cogenus $l$ to lower cogenus.  Fix $0 < l' < l$.  Recall the decomposition of $\tau = \left(\begin{smallmatrix} \tau_1 & z \\ \rT z & \tau_2 \end{smallmatrix}\right) \in \HS_g$, where $\tau_2$ has size $l'\times l'$.  We refine this decomposition as follows:
\begin{gather*}
  \tau
=
  \begin{pmatrix}
  \tau_{11} & z_{11} & z_{12} \\
  \rT z_{11} & \tau_{12} & z_{22} \\
  \rT z_{12} & \rT z_{22} & \tau_{2}
  \end{pmatrix}
\text{.}
\end{gather*}
Here $\tau_{11}$ has size $(g-l)\times (g-l)$ and $\tau_{12}$ has size $(l-l')\times (l-l')$.  The off diagonal matrices $z_{11}$, $z_{12}$, and $z_{22}$ have size $(g - l) \times (l-l')$, $(g - l) \times l'$, and $(l-l') \times l'$, respectively.  In addition, write $z_1$ for the $(g - l) \times l$ matrix $(z_{11}\; z_{12})$.

Given a formal Fourier-Jacobi expansion
\begin{gather*}
  f(\tau)
=
  \sum_{m \in \MatT{l}(\QQ)}
  \phi_m(\tau_{11}, z_1)\,
  e\!\left(m
   \left(\begin{matrix} \tau_{12} & z_{22} \\ \rT z_{22} & \tau_2   \end{matrix}\right) \right)
%\text{,}
\end{gather*}
of cogenus $l$, we define formal Fourier-Jacobi coefficients of index $m' \in \MatT{l'}(\QQ)$ by:
\begin{gather}
\label{eq:def:fourier-jacobi-coefficients-of-formal-expansions}
  \psi_{m'}(\tau_1, z)
=
  \sum_{\substack{n\in \Mat{l-l'}(\QQ) \\ r\in \Mat{l-l',l'}(\QQ)}}
   \phi_{\kzxz{n}{r}{\rT r}{m'}}(\tau_{11}, z_1)\,
   e\!\left( n\tau_{12} +2 r \rT z_{22} \right)
\text{.}
\end{gather}

\begin{lemma}
\label{la:formal-siegel-phi-operator}
Let $f \in \rmF\rmM^{(g,l)}_k$ be a symmetric formal Fourier-Jacobi series.  Fix $l' = 1$.  Then $\psi_0$ defined in~\eqref{eq:def:fourier-jacobi-coefficients-of-formal-expansions} is a symmetric Fourier-Jacobi series of genus~$g-1$ and cogenus $l-1$.
\end{lemma}
\begin{proof}
Consider $\phi_m(\tau_{11}, z_1)$ with
\begin{gather*}
  m
=
  \begin{pmatrix}
    m'' & 0 \\
    0  & 0
  \end{pmatrix}
\text{.}
\end{gather*}
By general theory of Siegel-Jacobi forms, $\phi_m$ is constant in $z_{12}$.  Therefore, $\psi_0$ depends only on $\tau_1$.  It can be written as
\begin{gather*}
  \psi_0(\tau_1)
=
  \sum_{m'' \in \Mat{l-1}(\QQ)}
  \phi_{\left(\begin{smallmatrix}m'' & 0 \\ 0 & 0 \end{smallmatrix}\right)}(\tau_{11}, z_{11})\,
  e(m'' \tau_{12})
\text{.}
\end{gather*}
The symmetry condition for $\psi_0$ follows directly from the one of $f$, by applying transformations of the form
\begin{gather*}
  \begin{pmatrix}
  u'' & 0 \\
  0 & 1
  \end{pmatrix}
  \in \GL{g}(\ZZ)
\text{,}\quad
  u'' \in \GL{g-1}(\ZZ)
\text{.}
\end{gather*}
\end{proof}

\begin{lemma}
\label{la:formal-theta-expansion}
Let $f \in \rmF\rmM^{(g,l)}_k$ be a symmetric formal Fourier-Jacobi series.  Fix $l' = l-1$ and a positive definite $m' \in \MatT{l'}(\QQ)$.  The formal Fourier-Jacobi coefficient $\psi_{m'}$ of $f$ defined in~(\ref{eq:def:fourier-jacobi-coefficients-of-formal-expansions}) has a formal theta expansion
\begin{gather}
\label{eq:formal-theta-expansions}
  \psi_{m'}(\tau_1, z)
=
  \sum_{\mu'} h_{m', \mu'}(\tau_1)\,
  \theta^{(g-l')}_{m', \mu'}(\tau_1, z)
\text{,}
\end{gather}
where the sum over $\mu'$ runs through $\Mat{g-l',l'}(\ZZ) (2 m')^{-1} \slashdiv \Mat{g-l',l'}(\ZZ)$, and the theta functions $\theta^{(g-l')}_{m', \mu'}$ are defined by \eqref{eq:defthetamu}.  We have $(h_{m', \mu'})_{\mu'} \in \rmF\rmM^{(g-l')}_{k - \frac{l'}{2}}\big( \rho^{(g-l')}_{m'} \big)$.
\end{lemma}

\begin{remark}
An analogue of Lemma~\ref{la:formal-theta-expansion} can be proved without restrictions on $l'$.  In Section~\ref{sec:rigidity-scalar-valued}, we only need the case $l' = l-1$, and so we restrict to the present setting, in order to minimize technical effort.
\end{remark}

\begin{proof}[Proof of Lemma~\ref{la:formal-theta-expansion}]
The symmetry condition for the Fourier coefficients of $f$ allows us to define a formal theta expansion of the $\psi_{m'}$.  Indeed,
%the theta decomposition can be reduced to a statement on symmetry of Fourier coefficients, which holds by the assumption on $f$.
the symmetry of $f$ under matrices of the form
\begin{gather*}
  \begin{pmatrix}1 & 0\\ \rT \lambda & 1\end{pmatrix} \in \GL{g}(\Z)
\text{,}\quad
  \lambda \in \Mat{g-l',l'}(\Z)
\text{,}
\end{gather*}
implies the identity of formal series
\begin{gather*}
  \psi_{m'}(\tau_1, z)
=
  \sum_{\mu'} h_{m', \mu'}(\tau_1)\,
  \theta^{(g-l')}_{m', \mu'}(\tau_1, z)
\text{,}
\end{gather*}
where
\begin{gather*}
  c\big(h_{m', \mu'};\, n' - x m'  \rT x \big)
=
  c\big(\psi_{m'};\, n',\,x \big)
\end{gather*}
for any representative $x \in \Mat{g-l',l'}(\QQ)$ of $\mu'$.
We have to show that $(h_{m',\mu'})_{\mu'}$ is a symmetric formal Fourier-Jacobi series of  cogenus $1$ and type $\rho^{(g - l')}_{m'}$.

As a first step, we show symmetry of Fourier coefficients.
Symmetry of Fourier coefficients of $f$ implies that
\begin{gather*}
  c\big( \psi_{m'};\, n', r' \big)
=
  \det(u)^k\, c\big( \psi_{m'};\, \rT u n' u, \rT u r' \big)
\end{gather*}
for all $n' \in \MatT{g-l'}(\Q)$, $r' \in \Mat{g-l',l'}(\Q)$, and $u \in \GL{g-l'}(\Z)$.  By fixing a representative $x \in \Mat{g-l',l'}(\QQ)$ of $\mu'$, we can deduce the corresponding relation for the Fourier coefficients of $h_{m', \mu'}$.  For $n' \in \MatT{g-l'}(\QQ)$, we have
\begin{multline*}
  c\big(h_{m', \mu'};\, n' - x m'  \rT x \big)
=
  c\big(\psi_{m'};\, n',\,x \big)
\\
=
  \det(u)^k\,  c\big(\psi_{m'};\, \rT u n' u, \rT u x \big)
=
  \det(u)^k\, c\big(h_{m', \rT u \mu'};\, \rT u n' u - \rT u x m' \rT x u \big)
\text{.}
\end{multline*}
Note that $\rT u$ acts on $\mu'$ in accordance with the representation $\rho^{(g-l')}_{m'}$.

As a second step, we have to examine the coefficients $\psi_{m',\mu',n'}$ of $h_{m', \mu'}$ in the formal expansion
\begin{align}
\label{eq:formal-theta-expansion-fourier-jacobi-expansion}
  h_{m',\mu'}(\tau_1)
=
  \sum_{n'\in \Q_{\geq 0}}
  \psi_{m',\mu',n'}(\tau_{11},z_{11})\,
  e(n'\tau_{12})
\text{.}
\end{align}
In view of \eqref{eq:def:fourier-jacobi-coefficients-of-formal-expansions}, \eqref{eq:formal-theta-expansions}, and \eqref{eq:formal-theta-expansion-fourier-jacobi-expansion}, comparing the $m'$-th Fourier-Jacobi coefficient for $m'\in \MatT{l'}(\QQ)$, we obtain the identity of formal power series
\begin{multline}
  \sum_{\substack{ n\in \Z_{\geq 0} \\ r \in \Mat{1,l'}(\Q)}}
  \phi_{\kzxz{n}{r/2}{\rT r/2}{m'}}(\tau_{11}, z_1)\,
  e\!\left( n\tau_{12} + r \rT z_{22} \right)
\\
=
  \sum_{n'\in \Q_{\geq 0}}\sum_{\mu'}
  \Big(
   \psi_{m', \mu', n'}(\tau_{11}, z_{11})\, e(n' \tau_{12})
  \Big)\,
  \theta^{(g-l')}_{m', \mu'}(\tau_1, z)
\text{,}
\end{multline}
which uniquely determines the coefficients $\psi_{m',\mu',n'}$.
We write $\mu'=\left(\begin{smallmatrix} \mu_1'\\\mu_2'\end{smallmatrix}\right)$ where the first component $\mu_1'\in \Mat{g-l,l'}(\Z)(2m')^{-1} \slashdiv \Mat{g-l,l'}(\Z)$ and  $\mu_2'\in \Mat{1,l'}(\Z)(2m')^{-1} \slashdiv \Mat{1,l'}(\Z)$.
Then
\begin{align*}
&\theta^{(g-l')}_{m', \mu'}(\tau_1, z) \\
&= \sum_{\substack{x_1\in \mu_1'+\Mat{g-l,l'}(\Z)\\ x_2 \in \mu_2'+\Mat{1,l'}(\Z)}} e\Big( x_1 m' \rT x_1 \tau_{11} + 2x_1 m' (\rT x_2 \rT z_{11}+ \rT z_{12})\Big)
 e(x_2 m' \rT x_2 \tau_{12} + 2x_2 m' \rT z_{22}).
\end{align*}
Inserting this in the previous identity and comparing the coefficients at
$e\!\left( n\tau_{12} + r \rT z_{22} \right)$, we find that for all $\mu_2'$, all $r$ with $r (2 m')^{-1} \equiv \mu_2'\pmod{\Mat{1,l'}(\Z)}$,
and all $n\in \Z_\geq 0$ we have
\begin{align}
\nonumber
\phi_{\kzxz{n}{r/2}{\rT r/2}{m'}}(\tau_{11}, z_1)
&=
  \sum_{\mu_1'}
   \psi_{m', \left(\begin{smallmatrix} \mu_1'\\\mu_2'\end{smallmatrix}\right), n'}(\tau_{11}, z_{11})\,
\sum_{x_1}
e\Big( x_1 m' \rT x_1 \tau_{11} + 2x_1 m' (\rT x_2 \rT z_{11}+ \rT z_{12})\Big)\\
\label{eq:fourier-jacobi-of-theta-decomposition}
&=
  \sum_{\mu_1'}
   \psi_{m', \left(\begin{smallmatrix} \mu_1'\\\mu_2'\end{smallmatrix}\right), n'}(\tau_{11}, z_{11})\,
   \theta^{(g-l)}_{m',\mu_1'}(\tau_{11}, z_{12}+z_{11} x_2)
\text{.}
\end{align}
Here $x_2=r (2 m')^{-1}$ and $n'= n-x_2m'\rT x_2$. This can be viewed as a partial theta decomposition
of the holomorphic
Jacobi form on the left hand side. The linear independence of the theta functions $\theta^{(g-l)}_{m',\mu_1'}(\tau_{11}, z_{12})$ as functions in $z_{12}$
implies that the $\psi_{m', \mu', n'}$ are holomorphic.

To finish the proof, consider the action
of the Siegel-Jacobi group
$\Sp{2(g-l)}(\Z)\ltimes \Mat{g-l,l}(\Z)^2$ embedded into $\Sp{2(g-l')}(\ZZ) \ltimes \Mat{g-l', l'}(\ZZ)^2 \subset \Sp{2g}(\ZZ)$ on the above identity \eqref{eq:fourier-jacobi-of-theta-decomposition}.
The left hand side is invariant in weight $k$ and index $n$ by the assumption on $f$.
On the right hand side $\theta^{(g-l)}_{m', \mu_1'}$ transforms by the restriction of the dual of $\rho^{(g-l')}_{m'}$ to the genus $g-l$ Jacobi group.
Arguing as in \cite{ziegler-1989}, Section~3, this implies the transformation law for $(\psi_{m', \mu', n'})_{\mu'}$.
\end{proof}

\subsection{Asymptotic dimensions}

We now establish formulas for the asymptotic dimensions of $\rmF\rmM^{(g)}_{k}$.  Our main tools are the order filtration and the theta decomposition for Siegel-Jacobi forms. Recall that~$\rmJ^{(g - 1)}_{k, m}[o]$ denotes the space of Siegel Jacobi forms~$\phi$ of genus~$g-1$, weight~$k$, index~$m$, and vanishing order $\ord \phi \ge o$.
\begin{lemma}
\label{la:embedding-of-formal-fourier-jacobi-expansions}
For every~$g$ and every~$k$, there is a (non canonical) embedding of vector spaces
\begin{gather*}
  \rmF\rmM^{(g)}_{k}
\longhookrightarrow
  \prod_{m \ge 0}  \rmJ^{(g - 1)}_{k, m}[m]
\text{.}
\end{gather*}
\end{lemma}
\begin{proof}
Consider the graded ring associated to the decreasing filtration
\begin{gather*}
  \rmF\rmM^{(g)}_k
\supset
  \rmF\rmM^{(g)}_k[1]
\supset
  \cdots
\end{gather*}
of symmetric formal Fourier-Jacobi series by their vanishing order.
For each $m \ge 0$, choose a linear section
\begin{gather*}
  \ell_m
:\,
  \rmF\rmM^{(g)}_k[m] \slashdiv \rmF\rmM^{(g)}_k[m+1]
\lra
  \rmF\rmM^{(g)}_k[m]
\subseteq
  \rmF\rmM^{(g)}_k
%\text{.}
\end{gather*}
for the canonical projection.
Recursively define formal Fourier-Jacobi series $f_m$ by means of $f_0 = f$ and $f_m = f_{m - 1} - \ell_{m-1}(f_{m-1})$ for $m \ge 1$.  The map from symmetric formal Fourier-Jacobi series into the corresponding graded ring
\begin{align*}
  \rmF\rmM^{(g)}_k
&\longrightarrow
  \prod_{m \ge 0} \rmF\rmM^{(g)}_k[m] \slashdiv \rmF\rmM^{(g)}_k[m+1]
\\
  f
&\lmto
  \Big( f_m +\rmF\rmM^{(g)}_k[m+1]\Big)_{m \ge 0}
\end{align*}
is injective, because its kernel equals $\bigcap_m \rmF\rmM^{(g)}_k[m] = \{0\}$.

By mapping a symmetric formal Fourier-Jacobi series in $\rmF\rmM^{(g)}_k[m]$ to its $m$-th Fourier-Jacobi coefficient, we obtain maps
\begin{gather*}
  \rmF\rmM^{(g)}_k[m]
\lra
  \rmJ^{(g-1)}_{k, m}[m]
\text{,}
\end{gather*}
whose kernel, for given~$m$, is $\rmF\rmM^{(g)}_k[m + 1]$.  This means that the maps
\begin{gather*}
  \rmF\rmM^{(g)}_k[m] \slashdiv \rmF\rmM^{(g)}_k[m + 1]
\longhookrightarrow
  \rmJ^{(g-1)}_{k, m}[m]
\text{,}
\end{gather*}
are injective.  By combining them with the above injection, we obtain the statement.
\end{proof}

\begin{lemma}
\label{la:maximal-vanishing-relation-for-jacobi-forms}
If $m\in \Z$ with $m> \frac{4}{3} \frac{k}{\varrho_g}$, then $\dim \rmJ^{(g)}_{k, m}[m] = 0$.
\end{lemma}

\begin{proof}
Fix some $\phi \in \rmJ^{(g)}_{k, m}[m]$ with theta decomposition
\begin{gather*}
\phi(\tau,z)=  \sum_{\mu \in \frac{1}{2m} \ZZ^g \slashdiv \ZZ^g}
  h_\mu(\tau)\, \theta^{(g)}_{m, \mu}(\tau, z)
\text{.}
\end{gather*}
%For a modular form of genus $g$ for the Weil representation $\rho_m^{(g)$
Given $r \in \Z^g$ and a Jacobi form~$\psi$ of genus~$g$, set
\begin{gather*}
  \ord_{r}\,\psi
=
  \inf \big\{
  m' \in \QQ \,:\,
  \text{$\exists t \in \MatT{g}(\QQ)$
  such that
  $c(\psi;\, t, r) \ne 0$ and $t_{g, g} = m'$}
  \big\}
\text{.}
\end{gather*}
Observe that $\ord\, \psi\leq \ord_r\, \psi $ for all~$r$.  In analogy with the usual vanishing order, we have $\ord_r\, f \psi = \ord\,f + \ord_r\,\psi$ for any Siegel modular form~$f$ of genus $g$.

For $\frac{1}{2m}r \in \mu + \ZZ^g$, we have
\begin{gather*}
  \ord\, \phi
\le
  \ord_{r}\, \phi
=
  \ord\, h_\mu + \ord_{r}\, \theta^{(g)}_{m, \mu}
\text{.}
\end{gather*}
Further, for $\frac{1}{2m}r = \frac{1}{2m}(r_1, \ldots, r_g) \in \mu + \ZZ^g$, we have $\ord_r\, \theta^{(g)}_{m, \mu} \le \frac{1}{4m}r_g^2$.  By choosing $-m < r_g \le m$, we find that $\ord_r\, \theta^{(g)}_{m, \mu} \le \frac{m}{4}$.  By the hypothesis, we have $\ord\, \phi \ge m$, so that $\ord\, h_\mu \ge m - \frac{m}{4} = \frac{3}{4}m$.
Hence we find  $\ord \,h >\frac{k}{\varrho_g}$, which  implies that $h = 0$ by
Proposition~\ref{prop:slope-bound-vector-valued}.
% , if $k < \frac{3}{4} \varrho_g m$.
\end{proof}

\begin{theorem}[{Runge}]
\label{thm:finitely-generated-ring-of-jacobi-forms-runge}
Fix $\epsilon > 0$, and assume that $g \ge 2$.  Then the ring
\begin{gather*}
  \bigoplus_{\substack{k, m \in \ZZ \\ k \ge \epsilon m}}
  \rmJ^{(g)}_{k, m}
\end{gather*}
is finitely generated.%
\footnote{Note added in June 2022: In the recent preprint arXiv:2203.14583v1~\cite{botero-burgos-gil-holmes-de-jong-2022-preprint} it is shown that Runge's Theorem is incorrect. However, in the present paper this result is only used in the proof of Lemma~\ref{la:asymptotic-dimension-for-jacobi-forms}. As explained in Remark~\ref{rm:la:asymptotic-dimension-for-jacobi-forms}, this Lemma~\ref{la:asymptotic-dimension-for-jacobi-forms} can also be proved directly by an analytic estimate, making the proof indepentent of Theorem~\ref{thm:finitely-generated-ring-of-jacobi-forms-runge}.}
\end{theorem}
\begin{proof}
This is Theorem~5.5 in \cite{runge-1995}; see also Remark~3.8 therein.
\end{proof}

\begin{lemma}
\label{la:asymptotic-dimension-for-jacobi-forms}
Fix $\epsilon > 0$.  For $k \ge \epsilon m$ and positive~$k$, we have
\begin{gather*}
  \dim \rmJ^{(g)}_{k, m}
\ll_{\epsilon}
  k^{\frac{g(g+1)}{2}} (m^g + 1)
\ll_{\epsilon}
  k^{\frac{g(g+3)}{2}}
\text{.}
\end{gather*}
\end{lemma}
\begin{proof}
If $g = 1$, one can prove this using weak Jacobi forms as in~\cite{eichler-zagier-1985}.  We restrict ourselves to the case $g \ge 2$, so that the assumptions of Theorem~\ref{thm:finitely-generated-ring-of-jacobi-forms-runge} are satisfied.

We apply Noether normalization to the bigraded ring $\bigoplus_{k \ge \epsilon m} \rmJ^{(g)}_{k, m}$.  This yields $d + 1 = \frac{g (g + 1)}{2} + 1$ Siegel modular forms $f_1, \ldots, f_{d + 1}$ of genus~$g$ and $d_\rmJ + 1$ Siegel-Jacobi forms $\phi_1, \ldots, \phi_{d_\rmJ + 1}$ that are algebraically independent.  We fix a basis $\psi_1, \ldots, \psi_{r_\rmJ}$ of $\bigoplus_{k \ge \epsilon m} \rmJ^{(g)}_{k, m}$ over
\begin{gather*}
  R
:=
  \CC[f_1, \ldots, f_{d+1}, \phi_1, \ldots, \phi_{d_\rmJ + 1}]
\text{.}
\end{gather*}

Theorem~5.1 of~\cite{runge-1995} identifies Jacobi forms with sections of line bundles over a projective variety of dimension $\frac{g(g + 1)}{2} + g$.  Therefore, we have that $d_\rmJ \le g$.

Write $k(f_i)$, $k(\phi_i)$, $m(\phi_i)$, $k(\psi_i)$, and $m(\psi_i)$ for the weight and index of the $f_i$, $\phi_i$, and $\psi_i$.
Moreover, for a tuple $a=(a_1, \ldots, a_{d_\rmJ + 1})$ of $d_\rmJ + 1$ integers write
\begin{align*}
m(a)&= a_1 m(\phi_1)+\dots +  a_{d_\rmJ + 1} m(\phi_{d_\rmJ + 1}),\\
k(a)&= a_1 k(\phi_1)+\dots +  a_{d_\rmJ + 1} k(\phi_{d_\rmJ + 1}).
\end{align*}
We denote the graded pieces of $R$ by $R_{k, m}$.  Note that $R_{k,0} \subseteq \rmM^{(g)}_k$, yielding bounds for $\dim R_{k,0}$.  For $m > 0$, we bound the graded dimensions as follows:
\begin{align*}
  \dim R_{k, m}
&=
  \sum_{\substack{a \in \ZZ_{\geq 0}^{d_\rmJ + 1} \\ m(a) = m}}
  \dim R_{k - k(a), 0}
\ll_\epsilon
  \sum_{\substack{a \in \ZZ_{\geq 0}^{d_\rmJ + 1} \\ m(a) = m}}
  (k - k(a) + 1)^d
  %{\substack{0 \le a_1, \ldots, a_{d_\rmJ + 1} \in \ZZ \\ \sum a_i m(\phi_i) = m}}
\\
&\ll_\epsilon
  (k + 1)^d\cdot
  \# \big\{ a \in \ZZ_{\geq 0}^{d_\rmJ + 1} :\; m(a) = m \big\}
\ll_\epsilon
  (k + 1)^d \big( m^{d_\rmJ} + 1 \big)
\text{.}
\end{align*}
We find
\begin{align*}
  \dim J^{(g)}_{k, m}
&=
  \sum_{i = 1}^{r_\rmJ}
  \dim R_{k - k(\psi_i), m - m(\psi_i)}
\\
&\ll_\epsilon
  \sum_{i = 1}^{r_\rmJ}
  \big( (k - k(\psi_i) + 1)^d \big)\,
  \big( (m - m(\psi_i))^{d_\rmJ} + 1 \big)
\ll_\epsilon
  k^d ( m^{d_\rmJ} + 1 )
\text{,}
\end{align*}
as desired.
\end{proof}

\begin{remark}
\label{rm:la:asymptotic-dimension-for-jacobi-forms}
An alternative, more analytic, proof of $\dim \rmJ^{(g)}_{k,m} \ll k^{\frac{g(g+3)}{2}}$ for $k > 0$ can be obtained by specializing the dimension bounds of \cite[Theorem 4]{wang-1999} to our situation, making use of $k \ge \epsilon m$.
\end{remark}

\begin{theorem}
\label{thm:asymptotic-dimension-for-formal-fourier-jacobi-expansions}
For every~$g$ and positive $k$, we have
\begin{gather*}
  \dim \rmF\rmM^{(g)}_{k}
\ll_g
  k^{\frac{g (g + 1)}{2}}
\text{.}
\end{gather*}
\end{theorem}
\begin{proof}
By Lemma~\ref{la:embedding-of-formal-fourier-jacobi-expansions}, we have an embedding
\begin{gather*}
  \rmF\rmM^{(g)}_{k}
\lhra
  \prod_{m \ge 0}  \rmJ^{(g - 1)}_{k, m}[m]
\text{.}
\end{gather*}
Using Lemma~\ref{la:maximal-vanishing-relation-for-jacobi-forms}, we find that
\begin{gather*}
  \dim \rmF\rmM^{(g)}_{k}
\le
  \sum_{m = 0}^{\frac{4 k}{3 \varrho_{g-1}}}
  \dim \rmJ^{(g - 1)}_{k, m}[m]
\text{.}
\end{gather*}

Lemma~\ref{la:asymptotic-dimension-for-jacobi-forms} provides us with a uniform estimate for the dimension of spaces of Jacobi forms that occur.  On plugging this in, we find the result:
\begin{gather*}
  \dim \rmF\rmM^{(g)}_{k}
\ll_g
  \sum_{m = 0}^{\frac{4 k}{3 \varrho_{g-1}}}
  k^{\frac{(g - 1) (g + 2)}{2}}
\ll_g
  k^{\frac{g (g + 1)}{2}}
\text{.}
\end{gather*}
\end{proof}

\section{Formal Fourier-Jacobi Expansions as an Algebra Extension}
\label{sec:purely-transcendental-extension}

In this section, we show that the ring of holomorphic Siegel modular forms is algebraically closed in
the ring of formal Fourier-Jacobi expansions.

\subsection{Geometry of Siegel modular varieties}

%We write $Y_g=\Sp{2g}(\Z)\bs \H_g$.
Let $X_g$ be a toroidal compactification of the Siegel modular variety
$Y_g=\Sp{2g}(\Z)\bs \H_g$ associated to an $\GL{g}(\Z)$-admissible
cone decomposition of the space of positive semi-definite symmetric
bilinear forms on $\R^g$, which is non-singular in the orbifold sense,
see e.g.~\cite{ash-etal-1975}, \cite{namikawa-1980} for details. We
write $\partial Y_g=X_g\setminus Y_g$ for the boundary divisor. There
is a natural map
\[
\pi :X_g\longrightarrow \overline{Y}_g
\]
to the Satake
compactification $\overline{Y}_g$ of $Y_g$.  The stratification of the
Satake boundary into Siegel modular varieties of lower genus induces a
stratification of $\partial Y_g$. For our argument, we will mainly need the boundary stratum
of genus $g-1$ (see Remark \ref{rem:ggeq3}).
For completeness we briefly recall its description, see
\cite{bruinier-van-der-geer-harder-zagier-2008}, Chapter~3.11.
%\cite{kramer-1995}, \cite{bruinier-van-der-geer-harder-zagier-2008}, Chapter~3.11  for details).
%  and let $X_g$ be its toroidal compactification
%(see \cite{Rappoport-et-al?} or \cite{bruinier-van-der-geer-harder-zagier-2008}, Chapter~3.11  for details).
%A partial toroidal compactification of $Y_g$ by a ``genus $g-1$ boundary stratum'' can be described as follows, see \cite{bruinier-van-der-geer-harder-zagier-2008}, Chapter~3.11.
%See also \cite{kramer-1995}.

We fix a splitting of $\tau\in \H_g$ as in \eqref{eq:tausplit} with $l=1$. For $c>0$ we consider the subset
\[
U_{g,c}=\left\{\tau=\zxz{\tau_1}{z}{\rT z}{\tau_2}\in \H_g:\; \Im(\tau_2) -\Im(\rT z)\Im(\tau_1)^{-1}\Im(z)>c\right\}
\]
of the Siegel upper half plane. The Klingen parabolic subgroup
\[
P=\left\{\begin{pmatrix} a & 0 & b & *\\
* & \pm 1 & * &*\\
c & 0 & d & *\\
0 & 0 & 0 & \pm 1\end{pmatrix}
\in \Sp{2g}(\Z) :\; \abcd\in \Sp{2(g-1)}(\Z)\right\}
\]
acts on $U_{g,c}$. If $c$ is sufficiently large, we obtain an embedding
\begin{align}
\label{eq:boundchart}
P\bs U_{g,c}\longrightarrow \Sp{2g}(\Z)\bs \H_g.
\end{align}
We consider the map
\[
U_{g,c}\longrightarrow \H_{g-1}\times \C^{g-1}\times \C^\times,\quad \tau\mapsto (\tau_1,z,q_2),
\]
where $q_2=e^{2\pi i\tau_2}$.
%The subgroup of $P$ of matrices of the form
%\[
%\begin{pmatrix} 1_{g-1} & 0 & 0 & 0 \\
%0 &  1 & 0 & b\\
%0 & 0 & 1_{g-1} & 0\\
%%\]
%with $b\in \Z$ acts trivially on the image
There is an induced action of $P$ on the image, which extends to $\H_{g-1}\times\C^{g-1}\times \C$.  It gives rise to a map
\begin{align}
\label{eq:partcomp}
P\bs U_{g,c} \longrightarrow P\bs (\H_{g-1}\times\C^{g-1}\times \C^\times),
\end{align}
which is biholomorphic onto its image. The quotient of the
boundary divisor $\H_{g-1}\times\C^{g-1}\times \{0\}$ by $P$ is given
by the universal principally polarized abelian variety
\begin{gather*}
\calX_{g-1}/\{\pm 1\}=\Sp{2(g-1)}(\Z)\ltimes \Z^{2(g-1)}\bs
(\H_{g-1}\times \C^{g-1})
\end{gather*}
of dimension $g-1$ modulo $\pm 1$.
We obtain a partial compactification of
$P\bs U_{g,c}$ by taking the closure of the image under the map \eqref{eq:partcomp}.
Glueing
this partial compactification onto $Y_g$ by means of \eqref{eq:boundchart}, we get
the partial compactification
\begin{align}
\label{eq:globalpartcomp}
X_g^{(1)} = Y_g\sqcup \calX_{g-1}/\{\pm 1\}.% \subset X_g.
\end{align}
The genus $g-1$ boundary stratum of $X_g$ is given by
$\calX_{g-1}/\{\pm 1\}$. The natural map from $\calX_{g-1}/\{\pm 1\}$
to the genus $g-1$ boundary stratum of the Satake compactification
$\overline{Y}_{g}$ is induced by the projection $\H_{g-1}\times
\C^{g-1}\times \{0\}\to \H_{g-1}$.

The local ring $\calO_{(\tau_1,z)}$ of $\Sp{2(g-1)}(\Z)\ltimes \Z^{2(g-1)}\bs
(\H_{g-1}\times \C^{g-1})$ at a point $(\tau_1,z)$ is given by the ring of invariants
$\CC\llbracket \td\tau_1 - \tau_1,\, \td z_1 - z_1 \rrbracket^G$ of the ring of convergent power series at $(\tau_1,z)$ under the action of the (finite) stabilizer
$G\subset \Sp{2(g-1)}(\Z)\ltimes \Z^{2(g-1)}$ of $(\tau_1,z)$.
The local ring $\calO_{(\tau_1,z,0)}$
of $X_g^{(1)}$ at a boundary point $(\tau_1,z,0)\in \calX_{g-1}$ is given
by the local ring of the quotient
$P\bs (\H_{g-1}\times\C^{g-1}\times \C)$ at $(\tau_1,z,0)$.
%If the boundary point is not an elliptic fixed point,
It is isomorphic to the ring of convergent power series $\calO_{(\tau_1,z)}\{q_2\}$ over
%the local ring
$\calO_{(\tau_1,z)}$.
% of $P\bs (\H_{g-1}\times\C^{g-1})$ at the point $(\tau_1,z)$.
The completion of $\calO_{(\tau_1,z,0)}$ at its maximal ideal is  the ring of formal power series
\begin{gather}
\label{eq:bndloc}
  \wht\calO_{(\tau_1,z,0)}
\cong
  \CC\llbracket \td\tau_1 - \tau_1,\, \td z_1 - z_1 \rrbracket^G
  \llbracket q_2 \rrbracket
\text{.}
\end{gather}
In particular, (formal) Fourier-Jacobi expansions of cogenus~$1$ define elements of these local rings.
Similarly, (formal) Fourier-Jacobi expansions of arbitrary cogenus~$l$
define elements of (completed) local rings of the genus $g-l$ boundary
stratum of $X_g$.

\begin{proposition}
\label{prop:effective-divisors-intersect-the-boundary-new}
Assume that $g\geq 2$, and let $D$ be a prime divisor on $X_g$.
Let $U\subset X_g$ be an open neighborhood of the boundary $\partial Y_g$.
Then $D\cap U$ is a non-trivial divisor on $U$.
\end{proposition}
\begin{proof}

If $D$ is supported on the boundary $\partial Y_g$, we have nothing to show. So we assume that
$D$ is not supported on the boundary.

%Let $\overline{Y}_g$ be the Satake compactification of $Y_g$ and let  be the natural map.
By our assumption, the pushforward $D'=\pi^*(D)$ under the natural map $\pi:X_g\to \overline{Y}_g$ is a prime divisor on $\overline{Y}_g$.
We employ the fact that $\Pic(\overline{Y}_g) \otimes \QQ = \QQ \cL$, where $\cL$ is the class of the Hodge bundle
(see \cite{borel-wallach-1980}, \cite{grushevsky-2009}).
Hence there is a positive integer~$n$ and a holomorphic Siegel modular form~$f$ of weight $k>0$
such that $\div\, f = n D'$ on $\overline{Y}_g$.
% Note that $\Pic(X_g)_\Q$ has dimension $\geq 2$ since there are extra parts generated by boundary divisors. For instance, $\Pic(X_g^{(g)})_\Q=\Q\cL\oplus\Q boundary$.
The restriction
%$\Phi(f)$
of $f$ to the boundary of $\overline{Y}_g$, that is, the image of $f$ under the $\Phi$-operator,
is a Siegel modular form of weight $k$ and genus $g-1\geq 1$.
It must vanish at some point of the Satake boundary, and therefore
$D'$ has non-trivial intersection with the
Satake boundary.
Consequently, $D'\cap \pi(U)$ is a non-trivial divisor on $\pi(U)$.
This implies the assertion.
%
% Old:
\begin{comment}
We employ the fact that $\Pic(Y_g) \otimes \QQ = \QQ \cL$, where $\cL$ is the class of the Hodge bundle
(see \cite{borel-wallach-1980}, \cite{grushevsky-2009}).
Hence there is a positive integer~$n$ and a holomorphic Siegel modular form~$f$ of weight $k>0$
such that $\div\, f = n D$ on $Y_g$.
The pullback
of $f$ to the boundary of the Satake compactification $\overline{Y}_g$ of $Y_g$
is a Siegel modular form of weight $k$ and genus $g-1\geq 1$.
It must vanish at some point of the Satake boundary, and therefore $f$ must vanish at some
point of the Satake boundary.
Hence $D$ intersects
Since $Y_g$ is normal, $f$ must in fact vanish on a divisor in any open neighborhood $U'$ of the Satake boundary.
Consequently, $D\cap U'$ is a non-trivial divisor on $U'$. This implies the assertion.
\end{comment}
\end{proof}

\subsection{Algebraic relations of formal Fourier-Jacobi series}

We need the following result from commutative algebra.
\begin{proposition}
\label{prop:ca}
Let $A$ be a local integral domain, and let $\hat A $ be the completion of $A$.
If $A$ is henselian and excellent, then $A$ is algebraically closed in $\hat A$.
\end{proposition}

\begin{proof}
If $A$ is an excellent local integral domain (not necessarily henselian), then its henselization can be described as the algebraic closure of $A$ in $\hat A$, see e.g.~\cite{freitag-kiehl-1988} pp.~16, or \cite{stacks-project} Example~16.13.3. This implies the assertion.
\end{proof}

\begin{lemma}
\label{la:local-convergence-of-formal-fourier-jacobi-expansions}
Let $Q=\sum_{i=0}^d a_i X^i\in \rmM_\bullet^{(g)}[X]$ be a nonzero polynomial of degree $d$ with coefficients $a_i\in \rmM_{k_0+(d-i)k}$, and let
\begin{gather*}
f=\sum_{m}\phi_m(\tau_1,z)q_2^m\in \rmF\rmM^{(g)}_k
\end{gather*}
be a formal Fourier-Jacobi expansion of cogenus~$1$ such that $Q(f)=0$. Then $f$ converges absolutely in an open neighborhood of the boundary divisor of $X_g$ and defines a holomorphic function there.
\end{lemma}

\begin{proof}
Let $(\tau_1,z,0)\in X_g^{(1)}$ be a boundary point
%which is not an elliptic fixed point
as in \eqref{eq:bndloc}. The polynomial $Q$ defines a polynomial
in $\calO_{(\tau_1,z,0)}[X]$, and $f$ defines an element of $\hat\calO_{(\tau_1,z,0)}$ which is algebraic over $\calO_{(\tau_1,z,0)}$ by hypothesis.
The local ring $\calO_{(\tau_1,z,0)}$ is henselian (a consequence of the Weierstrass preparation theorem)
and excellent (see e.g.~\cite{matsumura-1980} Theorem~102). Hence Proposition~\ref{prop:ca} implies that $f$ converges in a neighborhood of $(\tau_1,z,0)$. Varying the boundary point,
%and using the fact that the set of elliptic fixed points is discrete
we find that $f$ converges in a neighborhood of the whole
boundary stratum of $X_g^{(1)}$. The same argument applies to the boundary strata of smaller genus.  This proves the proposition.

\end{proof}

\begin{lemma}
\label{lem:coversplit}
Let $W\subset \C^N$ be a domain and let $Q(\tau,X)\in \calO(W)[X]$ be a monic irreducible polynomial with discriminant $\Delta_Q\in \calO(W)$.
Let $V\subset W$ be an open subset that has non-trivial intersection with every irreducible component of the divisor $D=\div(\Delta_Q)$.
%such that $V\cap D_0\neq \emptyset $ for every irreducible component $D_0$ of the divisor $D=\div(\Delta_Q)$.
If $f$ is a holomorphic function on $V$ satisfying $Q(\tau,f(\tau))=0$ on $V$, then $f$ has a holomorphic continuation to $W$.
\end{lemma}

\begin{proof}
Let $\tilde W=\{(\tau,X)\in W\times \C:\; Q(\tau, X)=0\}$ be the
analytic hypersurface defined by $Q$. The projection to the first
coordinate defines a branched covering
\begin{gather*}
p_1:\tilde W\longrightarrow W
\end{gather*}
of degree $\deg(Q)$.
The branching locus in $W$ is the divisor $D$.
Since $Q$ is irreducible, the corresponding unbranched cover $\tilde W\setminus p_1^{-1}(D)\to W\setminus D$ is connected and the automorphism group $\Aut(\tilde W/W)$ of the covering acts transitively on the fibers.

Over the open subset $V\subset W$ the map $s: V\to \tilde W, \tau\mapsto (\tau,f(\tau))$ defines a holomorphic section, and using the projection $p_2: \tilde W\to \C$ to the second factor, we have
$f= p_2\circ s$. For every $\sigma\in \Aut(\tilde W/W)$, the composition
\begin{gather*}
f_\sigma=p_2\circ \sigma \circ s
\end{gather*}
defines a holomorphic function on $V$ satisfying $Q(\tau,f_\sigma(\tau))=0$. Since $\Aut(\tilde W/W)$ acts transitively on the fibers, we find that $Q(\tau,X)$ splits completely into linear factors over $V$.

We now show that for every $a\in W$ the localized polynomial $Q_a(\tau,X)\in \calO_a[X]$ with coefficients in the local ring at $a$ splits completely into linear factors. We have just shown this for all $a\in V$. Moreover, it is  clear for $a\in W\setminus D$ by the theorem of implicit functions.
Next we show it for all points in the regular locus $D^{reg}$ of the branch divisor $D$.

Let $D_0$ be an irreducible component of $D^{reg}$ and let $a\in D_0$. Choosing holomorphic coordinates appropriately, we may assume that there is a small polycylinder $U\subset W$ around $a$ in which $D_0\cap U=D\cap U=\{\tau\in U:\; \tau_1=0\}$. Then according to \cite{grauert-remmert-1958},
Satz~10 and Hilfssatz~2 in \S2.5, every irreducible component of $\tilde W\cap p_1^{-1}(U)$ is a winding covering, that is, isomorphic to a covering of the form $\{(\tau,X):\; X^c-\tau_1=0\}$, where $c$ is the covering degree. This implies that $Q_a(\tau,X)\in \calO_a[X]$ factors into linear factors if and only if
$Q_b(\tau,X)\in \calO_b[X]$ factors into linear factors for all $b$ in a full open neighborhood of $a$ in $D_0$.
Hence
\begin{align*}
U_1 &= \{ a\in D_0 :\; \text{$Q_a(\tau,X)$ decomposes into linear factors in $\calO_a[X]$}\},\\
U_2 &= \{ a\in D_0 :\; \text{$Q_a(\tau,X)$ does not decompose into linear factors in $\calO_a[X]$}\}
\end{align*}
are disjoint {\em open} subsets of $D_0 $ whose union is $D_0$.
Since $D_0 $ is connected, and since $U_1\cap V\neq \emptyset$ we find that
$U_1=D_0 $.

Since the singular locus of $D$ is a closed analytic subset of $W$ of codimension $\geq 2$, Hartogs' Theorem (see e.g.~Theorem 2.5 in Chapter VI of \cite{fritsche-grauert-2002}) implies that
$Q_a(\tau,X)$ splits completely into linear factors for all $a\in W$.
Therefore, by  \cite{fritsche-grauert-2002}, Proposition 4.10 in Chapter III, $\tilde W\setminus p_1^{-1}(D)$ decomposes into $d$ connected components which are biholomorphically mapped onto $W\setminus D$ by $p_1$. Thereby we obtain the desired continuation of $f$ to a holomorphic function on $W\setminus D$ solving the polynomial $Q(\tau,X)$. Since the continuation is locally bounded, it extends to all of $W$.
\end{proof}

\begin{theorem}
\label{thm:establish-holomorphicity}
Let $Q=\sum_{i=0}^d a_i X^i\in \rmM_\bullet^{(g)}[X]$ be a nonzero polynomial of degree $d$ with coefficients
$a_i\in \rmM_{k_0+(d-i)k}$, and let
\begin{gather*}
  f
=
  \sum_{m} \phi_m(\tau_1,z)\,
  q_2^m
\in
  \rmF\rmM^{(g)}_k
\end{gather*}
be a formal Fourier-Jacobi expansion of cogenus~$1$ such that $Q(f)=0$. Then $f$ converges absolutely on $\HS_g$
and defines an element of  $\rmM_k^{(g)}$.
\end{theorem}

\begin{proof}
Without loss of generality we may assume that $Q$ is irreducible.
First, we assume that $Q$ is also monic (and therefore $k_0=0$).
%As before, we consider the toroidal compactification~$X_g$ of $Y_g$.
According to Lemma~\ref{la:local-convergence-of-formal-fourier-jacobi-expansions} there exists an open
neighborhood $U\subset X_g$ of the boundary divisor $\partial
Y_g\subset X_g$ on which $f$ converges absolutely. Hence $f$ defines a
holomorphic function of the inverse image $V\subset\H_g$ of $U$ under the natural
map $\H_g \to X_g$.

The discriminant $\Delta_Q$ of $Q$ is a holomorphic Siegel modular form of weight $d(d-1)k$.
According to Proposition~\ref{prop:effective-divisors-intersect-the-boundary-new}
every irreducible component of $D=\div \Delta_Q$ has non-trivial intersection with
$V$. Employing Lemma \ref{lem:coversplit}, we find that
$f$ has a holomorphic continuation to all of $\H_g$.

If the polynomial $Q$ has leading coefficient $a_d$ not equal to $1$,
then by a standard argument there is a monic polynomial $R\in
\rmM_\bullet^{(g)}[X]$ of degree $d$ such that $R(a_d\cdot f)=0$.
Replacing in the above argument $f$ by $h=a_d\cdot f$, we see that $h$
has a holomorphic continuation to $\H_g$. Therefore $f$ is a
meromorphic Siegel modular form which is holomorphic on $V$.  By
%Lemma~\ref{la:local-convergence-of-formal-fourier-jacobi-expansions} and
Proposition~\ref{prop:effective-divisors-intersect-the-boundary-new}, its
polar divisor must be trivial, and therefore $f$ is in fact
holomorphic on $\H_g$.

This implies that the formal Fourier-Jacobi expansion of $f$ converges on all of $\H_g$.
Since $\Mp{2g}(\Z)$ is generated by the
embedded Jacobi group $\widetilde{\Gamma}^{(g-1,1)}$ and the
embedded group $\GL{g}(\Z)$, we find that $f\in \rmM_k^{(g)}$.
\end{proof}

\begin{remark}
\label{rem:ggeq3}
If $g\geq 3$ then in the above proof the open neighborhood $U$ of the boundary $\partial Y_g$ can be replaced by an open neighborhood $U$ of the boundary of the partial compactification
$X_g^{(1)} = Y_g\sqcup \calX_{g-1}/\{\pm 1\}$. In fact, an inspection of the proof shows that the crucial point is that the analogue of Proposition~\ref{prop:effective-divisors-intersect-the-boundary-new} must hold. This follows from the fact that a holomorphic Siegel modular form of genus $g$ of positive weight vanishes at some point of the genus $g-1$ boundary stratum of the Satake compactification, since the Satake compactification is normal.
\end{remark}

\begin{corollary}
\label{cor:free-algebra-of-formal-fourier-jacobi-expansions}
The graded ring $\rmM^{(g)}_\bullet$ is algebraically closed in the graded ring $\rmF\rmM^{(g)}_\bullet$.
\end{corollary}

\section{Modularity of Symmetric Formal Fourier-Jacobi Series}
\label{sec:rigidity-scalar-valued}

We start this section with a direct consequence of Corollary~\ref{cor:free-algebra-of-formal-fourier-jacobi-expansions} and Theorem~\ref{thm:asymptotic-dimension-for-formal-fourier-jacobi-expansions}.
\begin{theorem}
\label{thm:rigidity-for-classical-forms-cogenus-1}
For any $g \ge 2$, we have
\begin{gather*}
  \rmF\rmM^{(g)}_\bullet
=
  \rmM^{(g)}_\bullet
\text{.}
\end{gather*}
\end{theorem}

\begin{proof}
Theorem~\ref{thm:asymptotic-dimension-for-formal-fourier-jacobi-expansions} shows that $\rmF\rmM^{(g)}_\bullet$ has finite rank as a graded $\rmM^{(g)}_\bullet$ module. Hence every element $f\in \rmF\rmM^{(g)}_k$ satisfies a non-trivial algebraic relation as in Theorem \ref{thm:establish-holomorphicity}, which then implies that $f$ belongs to $\rmM^{(g)}_\bullet$.
\end{proof}

To extend this to symmetric formal Fourier-Jacobi series of arbitrary type and cogenus~$1 \le l <g$, we apply induction on~$g$.  Our Main Theorem is a consequence of both of Lemma~\ref{la:rigidity-for-all-cogenera} and Lemma~\ref{la:rigidity-for-all-types}.

\begin{lemma}
\label{la:rigidity-for-all-cogenera}
Fix $g \ge 3$, and assume that $\rmF\rmM^{(g')}_\bullet(\rho) = \rmM^{(g')}_\bullet(\rho)$ holds for all $2 \le g' < g$ and for all representations~$\rho$ of $\Mp{2g'}(\ZZ)$ with finite index kernel.  Then we have $\rmF\rmM^{(g,l)}_\bullet = \rmM^{(g)}_\bullet$ for all $1 \le l < g$.
\end{lemma}

\begin{proof}
In Theorem~\ref{thm:rigidity-for-classical-forms-cogenus-1}, we have established that $\rmF\rmM^{(g,1)}_\bullet = \rmM^{(g)}_\bullet$.  We use induction on $l$ to establish all other cases.  That is, we now suppose that $1 < l < g$ and $\rmF\rmM^{(g,l-1)}_\bullet = \rmM^{(g)}_\bullet$.  We will argue that $\rmF\rmM^{(g,l)}_\bullet = \rmM^{(g)}_\bullet$.

We adopt notation of Section~\ref{ssec:siegel-phi-and-fourier-jacobi}.  In particular, we put $l'=l-1$ and fix a symmetric formal Fourier-Jacobi series of weight $k$ and cogenus~$l$
\begin{align*}
  f(\tau)
=
  \sum_{m \in \MatT{l}(\QQ)}
  \phi_m(\tau_{11}, z_1)\,
  e\!\left(m
   \left(\begin{matrix} \tau_{12} & z_{22} \\ \rT z_{22} & \tau_2   \end{matrix}\right) \right)
\text{.}
\end{align*}
We consider its formal Fourier-Jacobi coefficients $\psi_{m'}$ for $m' \in \MatT{l'}(\QQ)$ as in Equation~(\ref{eq:def:fourier-jacobi-coefficients-of-formal-expansions}).  We will show that the formal series
\begin{gather*}
  \sum_{m' \in \MatT{l'}(\QQ)}
  \psi_{m'}(\tau_1, z)\,
  e(m' \tau_2)
\end{gather*}
is a symmetric formal Fourier-Jacobi series of cogenus $l'$.  Symmetry of Fourier coefficients is immediate, and so we are reduced to establishing convergence of all $\psi_{m'}$, thereby proving that $\psi_{m'} \in \rmJ^{(g-l')}_{k, m'}$.

First, consider the case $\det\,m' = 0$.  If $m'$ is of the form
\begin{gather}
\label{eq:rigidity-for-all-cogenera:degenerate-jacobi-index}
  \begin{pmatrix}
  m'' & 0 \\
  0 & 0
  \end{pmatrix}
\text{,}\quad
  m'' \in \MatT{l'-1}(\QQ)
\text{,}
\end{gather}
then Lemma~\ref{la:formal-siegel-phi-operator} in conjunction with our assumptions implies that $\psi_{m'}$ converges.  We reduce the case of general degenerate~$m'$ to the above one.  For every $m'$ with $\det\,m' = 0$ there is a $u' \in \GL{l'}(\ZZ)$ such that $\rT u' m' u'$ is of the form~\eqref{eq:rigidity-for-all-cogenera:degenerate-jacobi-index}.  Invariance of $f$ under the action of
\begin{gather*}
  \begin{pmatrix}
  1 & 0 \\
  0 & u'
  \end{pmatrix}
  \in \GL{g}(\ZZ)
\text{,}\quad
  u' \in \GL{l'}(\ZZ)
\end{gather*}
shows that
\begin{gather*}
  \psi_{\rT u' m' u'}(\tau_1, z )
=
  \det(u')^k\,
  \psi_{m'}\left(\tau_1,
    z  \begin{pmatrix}
        1 & 0 \\
        0 & \rT u'
      \end{pmatrix}
  \right)
\text{.}
\end{gather*}
This establishes the convergence of $\psi_{m'}$ in the case $\det\,m' = 0$.

Next, we consider the case of invertible $m'$.  We use Lemma~\ref{la:formal-theta-expansion} to represent $\psi_{m'}$ in terms of a vector valued symmetric formal Fourier-Jacobi series~$(h_{m', \mu'})_{\mu'}$ of genus~$g-1$.  By the assumptions, it converges and hence so does $\psi_{m'}$.
\end{proof}

\begin{lemma}
\label{la:modular-forms-separate-points}
Let $\rho$ be a finite dimensional representation of $\Mp{2g}(\ZZ)$ with finite index kernel.  Then there exists a $k \in \frac{1}{2}\ZZ$ such that $\rmM^{(g)}_{k}(\rho^\vee)$ separates points of $V(\rho)$ at every~$\tau \in \HS_g$ which is not an elliptic fixed point.  That is, for every such $\tau $ and every $v \in V(\rho)$ there is an $f \in \rmM^{(g)}_{k}(\rho^\vee)$ such that $f(\tau)(v) \ne 0$.
\end{lemma}

\begin{proof}
This is proved in Proposition~2.4 of \cite{bruinier-2015} for $g=2$.  The proof literally carries over to the case of arbitrary genus.
\end{proof}

\begin{lemma}
\label{la:rigidity-for-all-types}
Fix $g \ge 2$ and $0 < l < g$.  Assume that $\rmF\rmM^{(g,l)}_\bullet = \rmM^{(g)}_\bullet$.  Then $\rmF\rmM^{(g,l)}_\bullet(\rho) = \rmM^{(g)}_\bullet(\rho)$ for all finite dimensional representations $\rho$ of~$\Mp{2g}(\ZZ)$ with finite index kernel.
\end{lemma}

\begin{proof}
We proceed as in the proof of Theorem~1.2 of~\cite[]{bruinier-2015}:
%Fix $f \in \rmF\rmM^{(g)}_k(\rho)$.
Employing Lemma~\ref{la:modular-forms-separate-points}, choose $k'$ such that $\rmM^{(g)}_{k'}(\rho^\vee)$ separates points of $V(\rho)$ at every $\tau \in \HS_g$ which is not an elliptic fixed point.  Write $\langle \,,\, \rangle$ for the canonical bilinear pairing
\begin{gather*}
  \rmF\rmM^{(g,l)}_k(\rho) \times \rmM^{(g)}_{k'}(\rho^\vee)
\lra
  \rmF\rmM^{(g,l)}_{k+k'}
\end{gather*}
induced by the evaluation map $V(\rho) \times V(\rho^\vee) \rightarrow \CC$.

Fix $f \in \rmF\rmM^{(g,l)}_k(\rho)$, and for all $f'\in \rmM^{(g)}_{k'}(\rho^\vee)$ consider $\langle f, f' \rangle$, which by assumptions is a Siegel modular form of weight~$k + k'$.  This allows us to identify $f$ with a meromorphic Siegel modular form of weight~$k$.  By the choice of~$k'$, $\rmM_{k'}(\rho^\vee)$ separates points of $V(\rho)$, and therefore~$f$ has no singularities.
\end{proof}

\begin{theorem}
\label{thm:rigidity-general-case}
Suppose that $2 \le g$, $0 < l < g$, and $k \in \frac{1}{2}\ZZ$.  Let $\rho$ be a finite dimensional representation of~$\Mp{2g}(\ZZ)$ that factors through a finite quotient.  Then we have
\begin{gather*}
  \rmF\rmM^{(g,l)}_{k}(\rho)
=
  \rmM^{(g)}_{k}(\rho)
\text{.}
\end{gather*}
\end{theorem}
\begin{proof}
The assertion follows by combining Lemmas~\ref{la:rigidity-for-all-cogenera} and~\ref{la:rigidity-for-all-types} and Theorem~\ref{thm:rigidity-for-classical-forms-cogenus-1}.
\end{proof}

%\section{Possible Extensions}

\section{Applications and Possible Extensions}
\label{sec:extensions-of-this-work}

\subsection{Kudla's modularity conjecture}
\label{ssec:kudla-conjecture}

We briefly explain how Theorem~\ref{thm:rigidity-general-case} can be applied in the context of Kudla's conjecture on modularity of the generating series of special cycles on Shimura varieties associated with orthogonal groups \cite[Section 3, Problem 3]{kudla-2004}. For the case of genus~$2$ see also~\cite{raum-2013a,bruinier-2015}.

Let $(V,Q)$ be a quadratic space over $\Q$ of signature $(n,2)$, and write $(\cdot\mathop{,}\cdot)$ for the bilinear
form corresponding to $Q$.
The hermitian symmetric space associated with the orthogonal group of $V$ can be realized as
\begin{gather*}
D=\{ z\in V\otimes_\Q \C:\; \text{$(z,z)=0$ and $(z,\bar z)<0$}\} \slashdiv \C^\times.
\end{gather*}
This domain has two connected components. We fix one of them, and denote it by $D^+$.
Let $L\subset V$ be an even lattice and write $L'$ for the dual
lattice.  Let $\Gamma\subset \Orth{}(L)$ be a subgroup of finite index
which acts trivially on the discriminant group $L'/L$ and which takes
$D^+$ to itself.  The quotient
\begin{gather*}
X_\Gamma= \Gamma \mathop{\bs} D^+
\end{gather*}
has a structure as a quasi-projective algebraic variety of dimension $n$. It
has a canonical model defined over a cyclotomic extension
%$E_\Gamma$
of $\Q$.

For $1\leq g\leq n$ let $S_{L,g}$ be the complex vector space of functions $(L'/L)^g\to \C$. The group $\Mp{2g}(\Z)$ acts on $S_{L,g}$ through the Weil representation $\omega_{L,g}$.
%\[
%\omega_{L,g}: \Mp{2g}(\Z) \longrightarrow \GL{}(S_{L,g}),
%\]
%see e.g.~\cite{kudla-2004}.
For every positive semi-definite $t \in \MatT{g}(\Q)$ of rank~$r(t)$ and every $\varphi \in S_{L,g}$ there is a special cycle class $Z(t,\varphi)$ in the Chow group $\CH^{g}(X_\Gamma)_\C$
of codimension $g$ cycles on $X_\Gamma$ with complex coefficients, see \cite{kudla-1997}, \cite{kudla-2004}.  We denote by $Z(t)$ the element $\varphi\mapsto Z(t,\varphi)$ of $\Hom\big( S_{L,g}, \CH^{g}(X_\Gamma)_\C \big)$.

\begin{conjecture}[Kudla]
\label{conj:ku}
The formal generating series
\begin{gather*}
  A_g(\tau)
=
  \sum_{\substack{t \in \MatT{g}(\Q)\\ t\geq 0}} Z(t) \, q^t
\end{gather*}
with coefficients in $S_{L,g}^\vee\otimes_\C \CH^{g}(X_\Gamma)_\C$ is a Siegel modular form in $M^{(g)}_{1+n/2}(\omega_{L,g}^\vee)$ with values in $\CH^{g}(X_\Gamma)_\C$.
\end{conjecture}

The analogous statement for the cohomology classes in $H^{2g}(X_\Gamma)$ of the $Z(t)$ was proved by Kudla and Millson in \cite{kudla-millson-1990}.

\begin{theorem}
Conjecture \ref{conj:ku} is true.
\end{theorem}

\begin{proof}
For $m\in  \MatT{g-1}(\Q)$ positive semi-definite and $(\tau_1,z)\in \H\times \C^{g-1}$ denote by
\begin{align*}
\phi_m(\tau_1,z) = \sum_{\substack{n\in \Q_{\geq 0}\\ r\in \Mat{1,g-1}(\Q)}} Z\begin{pmatrix}
n & r/2\\\rT r/2 & m
\end{pmatrix} \cdot e(n\tau_1+r\rT z)
\end{align*}
the $m$-th formal Fourier-Jacobi coefficient of cogenus $g-1$.
It was proved by W.~Zhang in \cite{zhang-2009}
that $\phi_m(\tau_1,z)$ is a Jacobi form of weight~$1+n/2$ and index $m$ with values in  $\CH^{g}(X_\Gamma)_\C$, that is, an element of $J^{(1)}_{1+n/2,m}(\omega_{L,g}^\vee)\otimes_\C\CH^{g}(X_\Gamma)_\C$.
Hence $A_g(\tau)$ can be viewed as a formal Fourier-Jacobi series of cogenus $g-1$. Its coefficients trivially satisfy the symmetry condition, and therefore
\[
A_g(\tau) \in \rmF\rmM^{(g,g-1)}_{1+n/2}(\omega_{L,g}^\vee) \otimes_\C\CH^{g}(X_\Gamma)_\C.
\]
Consequently, the assertion follows from Theorem \ref{thm:rigidity-general-case}.
\end{proof}

\begin{corollary}
The subgroup
%$\CH_{\mathrm{sp}}^{g}(X_\Gamma)$
of $\CH^{g}(X_\Gamma)$ generated by the classes
$Z(t,\varphi)$ for $t\in \MatT{g}(\Q)$ posi\-tive semi-definite and $\varphi\in S_{L,g}$ has rank $\leq \dim\big( M_{1+n/2}^{(g)}(\omega_{L,g}^\vee) \big)$.
\end{corollary}

Note that it is not known in general whether the rank of $\CH^{g}(X_\Gamma)$ is finite.

\subsection{Vector valued factors of automorphy}
\label{ssec:extension-to-allvector-valued}

Every finite dimensional representation $\rho_\infty$ of the connected double cover of $\GL{g}(\CC)$ corresponds to a $K_\infty$-type of Siegel modular forms.  Most prominently, symmetric powers of the standard representation are included by this definition.  See van der Geer's introduction into the subject in~\cite{bruinier-van-der-geer-harder-zagier-2008} for details.  Along the very same lines as Lemma~\ref{la:rigidity-for-all-types}, one can prove that symmetric formal Fourier-Jacobi series of vector valued Siegel modular forms in this sense converge.

Indeed, let $\rho_\infty$ be a finite dimensional representation of the connected double cover of~$\GL{g}(\CC)$, and let $\rho$ be a representation of $\Mp{2g}(\ZZ)$ with finite index kernel.  Without further restriction, we may assume that $g \ge 2$.  Then we call a holomorphic function~$f :\, \HS_g \rightarrow V(\rho_\infty) \otimes V(\rho)$ a (doubly) vector valued Siegel modular form if, for every $\gamma \in \Mp{2g}(\ZZ)$, we have
\begin{gather*}
  f(\gamma \tau)
=
  \rho_\infty(c \tau + d) \rho(\gamma)\, f(\tau)
\text{.}
\end{gather*}
We write $\rmM^{(g)}(\rho_\infty,\, \rho)$ for the space of such functions.  There is an obvious analog $\rmF\rmM^{(g,l)}(\rho_\infty,\, \rho)$ in the formal setting.  We find that
\begin{gather}
 \rmF\rmM^{(g,l)}(\rho_\infty,\, \rho)
=
\rmM^{(g)}(\rho_\infty,\, \rho)
\text{,}
\end{gather}
if $0 < l < g$.

\subsection{Computation of Siegel modular forms}
\label{ssec:compuation}

Symmetric formal Fourier-Jacobi series have appeared in~\cite{ibukiyama-poor-yuen-2012,raum-2013a} in computations of Siegel modular forms and paramodular forms of genus~$2$.  Using the approach presented in this paper, we can formulate an algorithm to compute Siegel modular forms of arbitrary genus, weight, and type.

We define symmetric formal Fourier-Jacobi polynomials.  Given~$m_0 \in \QQ$, write $f = \sum_m \phi_m\, e(m \tau_2)$ for an element of $\bigoplus_{0 \le m < m_0}\! \rmJ^{(g-1)}_{k, m}(\rho)$.  Define its Fourier coefficients as in the case of symmetric formal Fourier-Jacobi series.  We say that $f$ is symmetric, if $c(f;\, t) = \omega^{2k}\rho(\rot(u), \omega)\, c(f;\, \rT u t u)$ for all $u\in \GL{g}(\Z)$ whenever the bottom right entry of both $t$ and $\rT u t u$ are less than~$m_0$.  Denote the space of symmetric formal Fourier-Jacobi polynomials by~$\rmF\rmM^{(g)}_{k, < m_0}(\rho)$.

One can reason as in~\cite{raum-2013a} that the natural map
\begin{gather*}
  \rmF\rmM^{(g)}_{k, < m_0}(\rho)
\lra
  \rmF\rmM^{(g)}_{k, < m'_0}(\rho)
\end{gather*}
is injective, if $m_0 > m'_0 > m^{(g)}_0$ for some $m^{(g)}_0$ that depends on $g$, $\rho$, and~$k$.  Further, Theorem~\ref{thm:rigidity-general-case} implies that
\begin{gather*}
  \rmM^{(g)}_k(\rho)
=
  \projlim_{0 < m_0 \in \ZZ}
  \rmF\rmM^{(g)}_{k, < m_0}(\rho)
\text{.}
\end{gather*}
Both statements together imply that $\rmM^{(g)}_k(\rho) = \rmF\rmM^{(g)}_{k, < m_0}(\rho)$ for some $m_0$.

An algorithm to compute (truncated) Fourier expansions of Siegel modular forms can thus be deduced along the lines of Section~8 of~\cite{raum-2013a}.  It would be interesting to have an implementation of this algorithm and to study its performance in practice.

\subsection{Hermitian modular forms and other kinds of automorphic forms}

Classical automorphic forms that behave very much like Siegel modular forms can be defined over imaginary quadratic number fields, totally real fields, and quaternion algebras that are ramified at infinity.  The subject is covered by Siegel~\cite{siegel-1951b} in one of the most general ways, and has been studied by Shimura in a series of papers (see for example~\cite{shimura-1978} and the references therein).  For an exposition specifically on Hermitian modular forms and quaternion modular forms see, e.g.,~\cite{braun-1949} and~\cite{krieg-1985}, respectively.  Symmetric formal Fourier-Jacobi series can be defined for all of them.  Our approach is applicable to several of these automorphic forms, but it does not seem suitable to cover all.  In fact, we expect that Hermitian modular forms over the integers of $\QQ(\sqrt{-1})$, $\QQ(\sqrt{-2})$, and $\QQ(\sqrt{-3})$, and quaternion modular forms over maximal orders of rational quaternion algebras of discriminant $2$ and $3$ can be dealt with in the same way as Siegel modular forms.

If modularity of symmetric formal Fourier-Jacobi series of Hermitian modular forms can be proved, the unitary Kudla conjecture should follow along the lines of Section~\ref{ssec:kudla-conjecture}.

%%%%%%%%%%%%%%%%%%%%%%%%%%%%%%%%%%%%%%%%%%%%%%%%%%
%%% BIBLIOGRAPHY

\renewbibmacro{in:}{}
\renewcommand{\bibfont}{\normalfont\small\raggedright}
\renewcommand{\baselinestretch}{.8}
\vspace{.8em}
\printbibliography%[heading=bibnumbered]
% In case publishers don't support biblatex, switch to amsref (or plain bibtex)
% \bibliography{bibliography.bib}

%%%%%%%%%%%%%%%%%%%%%%%%%%%%%%%%%%%%%%%%%%%%%%%%%%
%%% AFFILIATIONS

\addvspace{1em}
\titlerule[0.15em]\addvspace{0.5em}

{\noindent\small
%Jan Bruinier\\
Fachbereich Mathematik,
Technische Universit\"at Darmstadt,
Schlossgartenstra\ss e~7,
D\nbd 64289 Darmstadt, Germany\\
E-mail: \url{bruinier@mathematik.tu-darmstadt.de} \\
Homepage: \url{http://www.mathematik.tu-darmstadt.de/~bruinier} \\[1.5ex]
{\noindent\small
%Martin Raum\\
Max Planck Institute for Mathematics,
Vivatsgasse~7,
D-53111, Bonn, Germany\\
E-mail: \url{martin@raum-brothers.eu}\\
Homepage: \url{http://raum-brothers.eu/martin}%\\[1.5ex]
}

\end{document}